\documentclass[11pt,reqno]{amsart}

\usepackage{amsmath, amsthm, amssymb}
\usepackage{amsfonts}

\usepackage{hyperref}

\usepackage[ansinew]{inputenc}
\usepackage[dvips]{epsfig}
\usepackage{graphicx}
\usepackage[english]{babel}

\usepackage{thmtools}
\theoremstyle{plain}
\declaretheorem[title=Theorem, parent=section]{theorem}
\declaretheorem[title=Lemma,sibling=theorem]{lemma}
\declaretheorem[title=Proposition,sibling=theorem]{proposition}

\theoremstyle{definition}
\declaretheorem[title=Definition,sibling=theorem]{definition}
\declaretheorem[title=Remark,sibling=theorem]{remark}
\declaretheorem[title=Remark, numbered=no]{remark*}
\declaretheorem[title=Example, sibling=theorem]{example}
\declaretheorem[title=Assumption, numbered=no]{assumption*}

\numberwithin{equation}{section}

\usepackage[backgroundcolor=white, bordercolor=blue,
linecolor=blue]{todonotes}

\usepackage{dsfont}
\usepackage{bbm}

\newcommand{\N}{\mathbb{N}}
\newcommand{\R}{\mathbb{R}}

\newcommand{\Z}{\mathbb{Z}}

\newcommand{\cE}{\mathcal{E}}

\newcommand{\al}{\alpha}

\newcommand{\be}{\beta}

\newcommand{\gm}{\gamma}
\newcommand{\dl}{\delta}
\newcommand{\Dl}{\Delta}
\newcommand{\lm}{\lambda}
\newcommand{\Lm}{\Lambda}

\newcommand{\varep}{\varepsilon}

\newcommand{\sig}{\sigma}

\newcommand{\om}{\omega}
\newcommand{\Om}{\Omega}
\newcommand{\z}{\zeta}

\newcommand{\eps}{\varepsilon}

\newcommand{\loc}{\mathrm{loc}}

\newcommand{\pl}{\partial}

\DeclareMathOperator{\supp}{supp}

\DeclareMathOperator*{\osc}{osc}

\renewcommand{\d}{\textnormal{\,d}}

\newcommand{\average}{{\mathchoice {\kern1ex\vcenter{\hrule height.4pt
width 6pt depth0pt} \kern-9.7pt} {\kern1ex\vcenter{\hrule
height.4pt width 4.3pt depth0pt} \kern-7pt} {} {} }}
\newcommand{\dashint}{\average\int}

\def\Xint#1{\mathchoice
    {\XXint\displaystyle\textstyle{#1}}%
    {\XXint\textstyle\scriptstyle{#1}}%
    {\XXint\scriptstyle\scriptscriptstyle{#1}}%
    {\XXint\scriptscriptstyle\scriptscriptstyle{#1}}%
    \!\int}
\def\XXint#1#2#3{\setbox0=\hbox{$#1{#2#3}{\int}$}
    \vcenter{\hbox{$#2#3$}}\kern-0.5\wd0}
\def\bint{\Xint-}
\def\dashint{\Xint{\raise4pt\hbox to7pt{\hrulefill}}}
\def\tmint{\Xint{\raise0pt\hbox to6pt{\hrulefill}}}

\def\XXiint#1#2#3{\setbox0=\hbox{$#1{#2#3}{\iint}$}
    \vcenter{\hbox{$#2#3$}}\kern-0.5\wd0}

\begin{document}
\allowdisplaybreaks
\title[Time-insensitive nonlocal parabolic Harnack estimates]{Time-insensitive nonlocal parabolic Harnack estimates}

\author{Naian Liao}
 
\author{Marvin Weidner}

\address{Fachbereich Mathematik, Paris-Lodron-Universit\"at Salzburg, Hellbrunner Str, 24, 5020 Salzburg, Austria}
\email{naian.liao@plus.ac.at}

\address{Departament de Matem\`atiques i Inform\`atica, Universitat de Barcelona, Gran Via de les Corts Catalanes 585, 08007 Barcelona, Spain}
\email{mweidner@ub.edu}
\urladdr{https://sites.google.com/view/marvinweidner/}



\allowdisplaybreaks

\begin{abstract}
We establish new Harnack estimates that defy the waiting-time phenomenon for global solutions to nonlocal parabolic equations. Our technique allows us to consider general nonlocal operators with bounded measurable coefficients. Moreover, we show that a waiting-time is required for the nonlocal parabolic Harnack inequality when local solutions are considered.
\end{abstract}

\allowdisplaybreaks

\maketitle


\section{Introduction}

The Harnack inequality for nonnegative solutions to the heat equation
\begin{align*}
    \partial_t u - \Delta u = 0 ~~ \text{ in } (0,1] \times \R^d
\end{align*}
was proved independently by Hadamard \cite{Had54} and Pini \cite{Pin54}, and it can be stated as follows
\begin{align*}
    \sup_{(\theta_1 \tau, \theta_2 \tau] \times B_{\tau^{1/2}}(x_o)} u \le c \inf_{(\theta_3 \tau, \theta_4 \tau] \times B_{\tau^{1/2}}(x_o)} u \qquad \forall\, 0 < \tau < 1, ~~ \forall\, x_o \in \R^d,
\end{align*}
where $0 < \theta_1 < \theta_2 < \theta_3 < \theta_4 \le 1$ are arbitrary and the constant $c > 0$ depends only on $d$ and $\{\theta_i:\,i=1,2,3,4\}$. A deep result of Moser (see \cite{Mos64,Mos71}) shows that the Harnack inequality continues to hold for solutions to parabolic equations with bounded and measurable coefficients.

It is well-known that the constant $c > 0$ in the Harnack inequality must blow up if the distance between the two time intervals tends to zero, i.e. when $\theta_3 - \theta_2 \to 0$. This distance is referred to as ``waiting-time" or ``time-lag" in the literature. Such waiting-time is 
\textit{characteristic}
for second order linear parabolic equations and it is closely related to the Gaussian decay $\sim e^{-|x|^2}$ of the fundamental solution (see \cite[page 103]{Mos64}).

\subsection{Main results}

In this article, we prove that -- in stark contrast to the local case -- Harnack inequalities hold true \textit{without any waiting-time} for \textit{linear nonlocal} parabolic equations modeled upon the fractional Laplacian. 

To be precise, we are interested in equations of the form
\begin{equation}\label{Eq:1:1}
    \partial_t u(t,x) - \mathcal{L}_t u(t,x)=0\quad\text{in}\quad (0,T]\times \R^d
\end{equation}
for some $T>0$ and $d\in\N$,
where the integro-differential operator $\mathcal{L}_t$ is given by
\begin{align}
\label{eq:op}
    -\mathcal{L}_t u(t, x) = \text{p.v.}~\int_{\R^d} (u(t, x) - u(t, y)) K(t;x,y) \d y
\end{align}
for a measurable kernel $K: \R \times \R^d \times \R^d \to [0,\infty]$, which satisfies for some $0<\lm\le\Lm$ that
\begin{align}\label{eq:Kcomp}
\frac{\lambda}{|x-y|^{d+2s}} \le K(t;x,y) = K(t;y,x) \le \frac{\Lambda}{|x-y|^{d+2s}}.
\end{align}
These operators $\mathcal{L}_t$ are naturally associated with energy functionals and can therefore be regarded as nonlocal operators in divergence form with bounded, measurable, and uniformly elliptic coefficients \eqref{eq:Kcomp}. 

 Our main result is the following time-insensitive Harnack inequality for weak solutions to \eqref{Eq:1:1}.

\begin{theorem}[time-insensitive Harnack inequality]\label{Thm:0}
    Assume that the kernel $K$ satisfies \eqref{eq:Kcomp}. Let $u \ge 0$ be a global weak solution to \eqref{Eq:1:1} in $(0,T] \times \R^d$ in the sense of \autoref{Def:global-sol}.  Then, there exists a constant $c>1$ depending on the data $\{d,s,\lm,\Lm\}$, such that for any $\tau \in (0,T]$ and any $x_o \in \R^d$, we have
     \begin{align}
     \label{eq:thm-0-1}
\sup_{(\frac14\tau,\tau]\times B_{\tau^{1/2s}}(x_o)} u \le c\, \tau \int_{\R^d} \frac{u(\tau,x)}{(\tau^{1/2s}+|x-x_o|)^{d+2s} }\d x
     \end{align}
     and  
     \begin{align}
          \label{eq:thm-0-2}
        c\, \inf_{(\frac34\tau,\tau]\times  B_{\tau^{1/2s}}(x_o)} u \ge \tau \int_{\R^d} \frac{u(\tau,x)}{(\tau^{1/2s}+|x-x_o|)^{d+2s} }\d x.
     \end{align}
     Consequently, we have
          \begin{align}
               \label{eq:thm-0-3}
     \sup_{(\frac14\tau,\tau]\times B_{\tau^{1/2s}}(x_o)} u \le c \inf_{(\frac34\tau,\tau]\times  B_{\tau^{1/2s}}(x_o)} u
     \end{align}
     for some constant $c=c(d,s,\lm,\Lm)$.
\end{theorem}

The Harnack inequality \eqref{eq:thm-0-3} is time-insensitive since the supremum and infimum of solutions are comparable over time intervals with a nonempty intersection.
We emphasize that the first two estimates \eqref{eq:thm-0-1} and \eqref{eq:thm-0-2} are ``representation-like''; they are of purely nonlocal nature, since they involve global and local quantities at the same time. Both estimates defy the waiting-time phenomenon. Note that \eqref{eq:thm-0-2} does not follow from the Harnack inequality \eqref{eq:thm-0-3} since it contains a global integral and is therefore of particular interest. 

As a corollary, we have the following elliptic-type Harnack inequality: 

\begin{theorem}[elliptic-type Harnack inequality]\label{Thm:1}
     Assume that the kernel $K$ satisfies \eqref{eq:Kcomp}. Let $u \ge 0$ be a global weak solution to \eqref{Eq:1:1} in $(0,T] \times \R^d$ in the sense of \autoref{Def:global-sol}. Then, there exists a constant   $c>1$ depending on the data $\{d,s,\lm,\Lm\}$, such that for any $\tau \in (0,T]$ and any $x_o \in \R^d$, we have
         \begin{align*}
      \sup_{B_{\tau^{1/2s}}(x_o)} u(\tau,\cdot)\le c\, \tau \int_{\R^d} \frac{u(\tau,x)}{(\tau^{1/2s}+|x-x_o|)^{d+2s} }\d x
     \end{align*}
      and
     \begin{align*}
     c\, \inf_{B_{\tau^{1/2s}}(x_o)} u(\tau,\cdot)\ge \tau \int_{\R^d} \frac{u(\tau,x)}{(\tau^{1/2s}+|x-x_o|)^{d+2s} }\d x.
     \end{align*}
     Consequently, we have
     \begin{align*}
         \sup_{B_{\tau^{1/2s}}(x_o)} u(\tau,\cdot) \le c \inf_{B_{\tau^{1/2s}}(x_o)} u(\tau,\cdot)
     \end{align*}
     for some constant $c=c(d,s,\lm,\Lm)$.
\end{theorem}

Another direct consequence is a Liouville-type result.
\begin{theorem}
    Assume that the kernel $K$ satisfies \eqref{eq:Kcomp}. Let $u$ be a global weak solution to \eqref{Eq:1:1} in $(-\infty,T] \times \R^d$ in the sense of \autoref{Def:global-sol}. If $u\ge k$ or $u\le k$ for some $k\in\R$, then $u$ must be constant.
\end{theorem}
Such a property is again in stark contrast to the case of the heat equation where one-sided bounds of solutions are insufficient to claim that they are constant.

Parabolic Harnack inequalities for \textit{local} solutions to nonlocal equations, i.e.
\begin{align}
\label{eq:local-sol}
    \partial_t u - \mathcal{L}_t u = 0 ~~ \text{ in } (0,1] \times B_2
\end{align}
with $\mathcal{L}_t$ as in \eqref{eq:op}, \eqref{eq:Kcomp} have recently been established in \cite{KaWe24c} via variational techniques (see also \cite{BaLe02,ChKu03,CKW20} for a probabilistic approach). For every $\tau \in (0,1]$ it holds
\begin{align}
\label{eq:local-Harnack}
    c^{-1}\sup_{(\frac{2}{5}\tau,\frac{3}{5}\tau]\times B_{\tau^{1/2s}}} u \le  \int_{\frac{1}{5}\tau}^{\frac{3}{5}\tau} \int_{\R^d} \frac{u(t,x)}{(\tau^{1/2s}+|x|)^{d+2s} }\d x \d t \le c\,  \inf_{(\frac45\tau,\tau]\times  B_{\tau^{1/2s}}} u.
\end{align}
Note that the Harnack inequality for local solutions exhibits a waiting-time. In particular, all three estimates \eqref{eq:thm-0-1}, \eqref{eq:thm-0-2}, and \eqref{eq:thm-0-3} in \autoref{Thm:0} are time-insensitive versions of \eqref{eq:local-Harnack}.

In Subsection \ref{subsec:counterexample}, we prove that the \textit{waiting-time is necessary} for a Harnack inequality to hold true \textit{in the case of local solutions}. 
In fact, we construct an example, demonstrating that a time-insensitive Harnack inequality for local solutions is false. Hence, \autoref{Thm:0} describes a \textit{purely nonlocal phenomenon} that only applies to \textit{global solutions}. 

Intuitively, the reason for the time-insensitivity of global parabolic Harnack inequalities for nonlocal equations is the polynomial decay of the heat kernel. Indeed, the fundamental solution $(t,x,y) \mapsto p_t(x,y)$ to \eqref{Eq:1:1} with $\mathcal{L}_t$ as in \eqref{eq:op} and 
$K(t;x,y)=K(x,y)$ as in \eqref{eq:Kcomp} enjoys the following two-sided bounds
\begin{align}
\label{eq:heat-kernel-intro}
    C^{-1} \left( t^{-\frac{d}{2s}} \wedge \frac{t}{|x-y|^{d+2s}} \right) \le p_t(x,y) \le C \left( t^{-\frac{d}{2s}} \wedge \frac{t}{|x-y|^{d+2s}} \right),
\end{align}
see \cite{ChKu03,CKW22}. As a consequence, a straightforward computation reveals that $p$ (respectively its two-sided bound) satisfies \eqref{eq:thm-0-3}. This phenomenon is in stark contrast to the local equation case, where the Gaussian decay of the fundamental solution makes time-insensitivity impossible. 

In the case of the fractional heat equation $\pl_t u +(-\Dl)^{\frac12}u=0$, we give another heuristic proof of the time-insensitive Harnack estimate. Assume that a solution $u$ admits a smooth initial datum $u(0,\cdot)$. Then, applying Fourier transform we obtain 
$$
u(x,t)=\int_{\R^d} P(t,x-y)u(0,y)\d y\quad \text{with}\> P(t,x-y)=c(d)\frac{t}{(t^2+|x-y|^2)^{\frac{d+1}{2}}}.
$$
The Poisson kernel $P$, a harmonic function in the upper-half space, satisfies the two-sided bounds in \eqref{eq:heat-kernel-intro}. Harnack estimate \eqref{eq:thm-0-3} follows from the harmonicity of $u$ in view of the above representation, whereas \eqref{eq:thm-0-1} and \eqref{eq:thm-0-2} can be deduced from the representation formula and the semigroup property of the Poisson kernel.

Let us mention \cite[Theorem~8.2]{BSV17}, where the authors observed a version of the time-insensitive Harnack estimate for the fractional heat equation, i.e. in case $\mathcal{L}_t = -(-\Delta)^s$. In that paper, they suggested an approach via a representation formula for solutions to the Cauchy problem and then deduced a Harnack estimate from \eqref{eq:heat-kernel-intro}. However, the initial data are assumed to be integrable functions and to decay at infinity like $(1 + |x|)^{-d-2s}$, which unfortunately excludes the fundamental solution from their framework. 

In contrast,  \autoref{Thm:0} holds true for any global weak solution to \eqref{Eq:1:1} (see \autoref{Def:global-sol}) and is therefore new even for the fractional heat equation. More importantly, we go beyond that and establish the time-insensitive Harnack inequality for general nonlocal operators $\mathcal{L}_t$ with bounded and measurable coefficients.   

A key feature of our approach is that it keeps working for an \textit{even more general class} of kernels $K$, which might \textit{violate the lower bound} in \eqref{eq:Kcomp}. It is well-known that for such general class of kernels the matching two-sided heat kernel bounds \eqref{eq:heat-kernel-intro} are no longer valid (see \cite{BoSz05,CKW21,CKW22}).
Therefore, it is no longer possible to deduce a time-insensitive Harnack inequality directly from a representation formula for solutions to \eqref{Eq:1:1} and a genuinely new approach is needed to deal with this generality.
We present this extension of \autoref{Thm:0} in Subsection \ref{subsec:intro-extension} and explain our strategy of proof in Subsection \ref{subsec:strategy}.

\subsection{Extension of the main results}
\label{subsec:intro-extension}

Instead of the pointwise lower bound in \eqref{eq:Kcomp}, we consider the following more general coercivity condition
\begin{equation}\label{eq:coercive}
    \iint_{B_r(x_o)\times B_r(x_o)} |v(x) - v(y)|^2 K(t;x,y) \d y \d x \ge \lm\, [v]^2_{H^s(B_r(x_o))}
\end{equation}
for any $t\in \R$, $r > 0$, $x_o \in \R^d$, and any $v\in H^s(B_r(x_o))$. Such coercivity condition is natural since it implies that the corresponding nonlocal energy still verifies functional inequalities, such as a Poincar\'e--, Faber-Krahn--, and Sobolev inequality. 
Moreover, we consider the following assumption, which is sometimes referred to as (UJS) and prevents $K$ from possessing oscillatory long jumps
\begin{align}
    \label{eq:UJS}
    K(t;x,y) \le \Lambda\, \bint_{B_r(x)} K(t;z,y) \d z
\end{align}
for every $x,y \in \R^d$, $t \in \R$ and $r \le ( \frac{1}{4} \wedge \frac{|x-y|}{4} )$. 

It was shown in \cite{KaWe24c} that the nonlocal parabolic Harnack inequality \eqref{eq:local-Harnack} with time-lag remains true for solutions to \eqref{eq:local-sol} if $K$ satisfies the upper bound in \eqref{eq:Kcomp}, \eqref{eq:coercive} and \eqref{eq:UJS}. Moreover, it follows from the results in \cite{CKW20,CKW21,CKW22} that these conditions are (almost) sharp for the validity of the nonlocal parabolic Harnack inequality.

A key advantage of our technique is that it keeps working under these more general conditions on $K$. Indeed, we have the following theorem.

\begin{theorem}
\label{thm:general}
    Assume that the kernel $K$ satisfies the upper bound in \eqref{eq:Kcomp} and the coercivity condition \eqref{eq:coercive}. Let $u \ge 0$ be a global weak solution to \eqref{Eq:1:1} in $(0,T] \times \R^d$ in the sense of \autoref{Def:global-sol}. Then, there exists a constant   $c>1$ depending on the data $\{d,s,\lm,\Lm\}$, such that for any $\tau \in (0,T]$ and any $x_o \in \R^d$ the estimate \eqref{eq:thm-0-1} holds true.
\end{theorem}

We expect \eqref{eq:thm-0-2} and \eqref{eq:thm-0-3} to hold as well, if \eqref{eq:UJS} is assumed in addition. As already mentioned, in the general setup of \autoref{thm:general}, the two-sided heat kernel bounds \eqref{eq:heat-kernel-intro} do not hold true, which is a key difficulty when adapting the proof to more general kernels. Hence, it would be very interesting to establish \eqref{eq:thm-0-2} and \eqref{eq:thm-0-3} in such a general framework.

\subsection{Strategy of the proof}
\label{subsec:strategy}

The main ingredient in the proof of our time-insensitive Harnack inequality in \autoref{Thm:0} is the following weighted global $L^1$ estimate for nonnegative solutions to \eqref{Eq:1:1}. 
Indeed, we prove in \autoref{lemma:weighted-L1-time-insensitive} that for any $\tau,\, t \in(0, T]$ it holds
\begin{align}
\label{eq:time-insensitive-L1-intro}
    c^{-1} \int_{\R^d} \frac{u(\tau,y)}{(1 + |y|)^{d+2s}} \d y \le \int_{\R^d} \frac{u(t,y)}{(1 + |y|)^{d+2s}} \d y \le c \int_{\R^d} \frac{u(\tau,y)}{(1 + |y|)^{d+2s}} \d y.
\end{align}

Having at hand \eqref{eq:time-insensitive-L1-intro}, \autoref{Thm:0} immediately follows by a combination of the weak parabolic Harnack and local boundedness estimates which hold true for local solutions and were established in \cite{KaWe24c} (see \eqref{eq:local-Harnack}). The waiting-time in the local estimates can be eliminated due to the time-insensitivity of the intrinsically nonlocal estimate \eqref{eq:time-insensitive-L1-intro}.

 We emphasize that also \eqref{eq:time-insensitive-L1-intro} holds true for the general class of kernels satisfying merely the upper bound in \eqref{eq:Kcomp} and the coercivity condition \eqref{eq:coercive}. This indicates that time-insensitive Harnack estimates are not just an artifact of the polynomial decay of the fundamental solution, but are inherently linked with the nonlocal nature of the problem under consideration.

The main idea to prove the time-insensitive weighted $L^1$ estimate in \eqref{eq:time-insensitive-L1-intro} is to test the equation for $u$ with $\Psi$, which is a weak solution to the following dual Cauchy problem 
\begin{align*}
    \begin{cases}
        \partial_t \Psi + \mathcal{L}_t \Psi &= 0 ~~ \text{ in } (0,T) \times \R^d,\\
        \Psi(T,\cdot) &= (1 + |\cdot|)^{-d-2s} ~~ \text{ in } \R^d.
    \end{cases} 
\end{align*}
Then, the main work consists in showing that the function
\begin{align}
\label{eq:dual-product-constant}
    \tau \mapsto \int_{\R^d} u(\tau,x) \Psi(\tau,x) \d x \quad \text{is constant},
\end{align}
which will lead to \eqref{eq:time-insensitive-L1-intro}. Heuristically, if we could insert $\Psi$ as a testing function in the weak formulation of $u$, then we would have for any $t_1,\,t_2\in(0,T)$ that
\begin{align*}
        0 = \int_{\R^d} u  \Psi  \d x \bigg|_{t_1}^{t_2} - \int_{t_1}^{t_2} \int_{\R^d} u\cdot ( \partial_t \Psi + \mathcal{L}_t \Psi )\d x \d t = \int_{\R^d} u  \Psi  \d x \bigg|_{t_1}^{t_2},
\end{align*}
and then the result would follow. However, realizing this idea in a rigorous fashion requires a significant amount of work and involves local estimates of nonlocal energies in terms of $u$ and $\Psi$ coupled with a delicate covering argument. The spirit is that sharp local estimates will tell how solutions behave when we enlarge the domain and approach $\R^d$.

Amongst those local estimates, we would like to highlight the following novel energy decay estimate for supersolutions
\begin{align*}
\int_0^{\tau} \int_{B_{1}} \int_{B_{1}} |u(t,x) - u(t,y)| K(t;x,y)|x-y| \d y \d x \d t \le c \,\tau^{1-\frac{1}{q}} \bigg( \int_0^{2\tau} \int_{B_{2}} u^q \d x \d t\bigg)^{\frac1q},
\end{align*}
where $q>1$.
This estimate is a substantial ingredient in the proof of \eqref{eq:time-insensitive-L1-intro} and also leads to an improved variant of the weak Harnack estimate, which we discuss in Subsection \ref{S:weak-Harnack}. 

Note that it would be much easier to establish \eqref{eq:dual-product-constant} for operators with smooth coefficients, since in that case, it would be possible to evaluate $\mathcal{L}_t \Psi$ in a pointwise way, and therefore one would not have to deal with a double integral structure (see also \autoref{Rmk:L1-weight}). Moreover, since we do not assume any regularity of $K$, it is only known that $\Psi$ is $C^{\alpha}$ for some small $\alpha \in (0,1)$. It would be much easier to prove \eqref{eq:time-insensitive-L1-intro} if $\Psi$ were a smooth function.

Finally, once \eqref{eq:dual-product-constant} is established, \eqref{eq:time-insensitive-L1-intro} follows from the fact that $\Psi$ enjoys the following two-sided bounds (see \autoref{lemma:Psi-bounds})
\begin{align*}
    \frac{c^{-1}}{(1 + |x|)^{d+2s}} \le \Psi(t,x) \le \frac{c}{(1 + |x|)^{d+2s}}.
\end{align*}
The lower estimate directly follows from local estimates.
To prove the upper estimate for $\Psi$ we make use of a representation formula for $\Psi$ in terms of the fundamental solution and apply the upper estimate of the fundamental solution that was proved in \cite{KaWe23}. Interestingly, we were unable to find a proof of the representation formula for solutions to the nonlocal Cauchy problem with $L^2$ initial data in the literature, when $\mathcal{L}_t$ is time-dependent. Therefore, we give a rigorous proof in the appendix of this paper (see \autoref{prop:representation}), using entirely analytic arguments. We believe that this result might be of independent interest.

\subsection{An improved weak Harnack inequality}\label{S:weak-Harnack}

Another main result of this paper is the following improved weak Harnack inequality for nonnegative, {\it local} supersolutions.

\begin{theorem}[improved weak Harnack inequality]
\label{thm:improved-weak-Harnack}
Assume that the kernel $K$ satisfies the upper bound in \eqref{eq:Kcomp} and \eqref{eq:coercive}. Let $u \ge 0$ be a weak supersolution to \eqref{Eq:1:1} in $(t_o , t_o + 8R^{2s}) \times B_{8R}(x_o)$ in the sense of \autoref{Def:local-sol}.  Then, there exists $c>1$ depending on the data $\{d,s,\lm,\Lm\}$, such that
\begin{align*}
\sup_{(t_o, t_o + R^{2s})} \bint_{B_R(x_o)} u(\cdot,x) \d x \le c \, \inf_{(t_o + 2R^{2s}, t_o + 8R^{2s}) \times B_R(x_o)} u.
\end{align*}
\end{theorem}

Note that even for solutions, the improved weak Harnack inequality in \autoref{thm:improved-weak-Harnack} does not follow from the parabolic Harnack inequality \eqref{eq:local-Harnack}, since the solution property is not required for $t < t_o$.

\autoref{thm:improved-weak-Harnack} seems to be new in the context of nonlocal equations and has not even been established for the fractional Laplacian.
Results of similar type have been established in the context of local nonlinear equations, such as the porous medium equation (see \cite{DGV12}), and are important in the theory of initial traces.

\subsection{Further literature}

In the following, we give an incomplete overview of the literature on Harnack inequalities for nonlocal equations. First, let us mention \cite{Kas09}, \cite{DKP14}, \cite{DKP16}, \cite{Coz17}, \cite{CKW23} where the De Giorgi--Nash--Moser theory for nonlocal elliptic equations with bounded and measurable coefficients was developed, and H\"older regularity and Harnack inequalities were established. 

An analogous theory for nonlocal parabolic equations \eqref{Eq:1:1} was developed in \cite{CCV11}, \cite{FeKa13}, where H\"older regularity estimates and weak Harnack inequalities were derived (see also \cite{Str19,KaWe24a,KaWe22}). Despite these developments, a parabolic Harnack inequality was established only recently in \cite{KaWe24c}. In fact, establishing the two inequalities in \eqref{eq:local-Harnack} requires a fine understanding of the associated nonlocal tail terms, i.e. the global quantities in \eqref{eq:local-Harnack}. In contrast to elliptic equations, these tail terms are scaling critical for parabolic problems in the sense that they cannot be treated by standard perturbation arguments. Therefore it is much more delicate to incorporate them into suitable estimates. On a related note, it was observed in \cite{KaWe24c,Lia24c} that weak solutions to \eqref{eq:local-sol} are in general only continuous, but not H\"older continuous, unless additional assumptions on the tail terms are imposed. The situation is even worse for nonlocal kinetic equations, where a Harnack inequality is in general false, even for the fractional Laplacian (see \cite{KaWe24b}).

There also exist several contributions in the literature dealing with nonlinear nonlocal parabolic equations modeled upon the fractional $p$-Laplacian (see \cite{APT22,ByKi24,Lia24,Lia24b,Lia24c}). However, the validity of a Harnack inequality remains an interesting open problem in the field. 

Moreover, we refer to \cite{ChKu03,CKW20} for approaches to H\"older regularity estimates and Harnack inequalities for solutions to \eqref{eq:local-sol} that are based on probabilistic arguments. Let us mention \cite{CaSi09,ChDa16} for related results for nonlocal equations in non-divergence form.

For Li-Yau-type inequalities pertinent to the fractional heat equation, we refer to \cite{DeSi23} and \cite{WeZa23}, and to the discussions in \cite{BSV17}, \cite{DKSZ20}, and \cite{KaWe24c}.

Finally, let us mention that there are several classes of \textit{nonlinear} parabolic equations, such as the fast diffusion equation or equations modeled upon the parabolic $p$-Laplacian with $p < 2$, which satisfy a Harnack inequality without waiting-time (see \cite{DiB93,BoVa10,DGV12}). However, the underlying mechanism for these equations appears to be different from the one we describe in this article, as no globality enters, whereas the extinction of solutions could occur abruptly.

\subsection{Outline}

This article is structured as follows. In Sections \ref{sec:notation} and \ref{sec:tools} we introduce the weak solution concept, and collect and prove several auxiliary results that will be used in the sequel of the article. In Sections \ref{sec:local} and \ref{sec:Psi}, we prove several preparatory results. In Section \ref{sec:time-insensitive} we show the time-insensitive $L^1(\R^d;\mu)$ estimate, establish our main results \autoref{Thm:0} and \autoref{Thm:1}, and demonstrate that the waiting-time cannot be avoided for local solutions. Section \ref{sec:improved-weak-Harnack} is dedicated to establishing the improved weak Harnack inequality (see \autoref{thm:improved-weak-Harnack}). Finally, in Appendix \ref{sec:appendix}, we prove a representation formula for solutions to the nonlocal Cauchy problem in terms of the fundamental solution.

\subsection*{Acknowledgments}

Naian Liao was supported by the FWF-project P36272-N ``On the Stefan type problems".
Marvin Weidner has received funding from the European Research Council (ERC) under the Grant Agreement No 801867 and the Grant Agreement No 101123223 (SSNSD), and by the AEI project PID2021-125021NA-I00 (Spain).

\section{Notation and notion of weak solutions}
\label{sec:notation}

In this manuscript, we use the notation $a\wedge b\equiv \min\{a,b\}$ and $a\vee b\equiv \max\{a,b\}$. The ball with radius $r$ and center $x_o$ is denoted as $B_r(x_o)$; when the context is clear, we omit $x_o$. The fractional Sobolev space $H^s(\Om)$ is standard; see~\cite{KaWe24c}. Moreover, we consider a Radon measure $\mu$ defined by
\begin{equation*}
    \mu(E):=\int_{E}\frac{\d x}{(1+|x|)^{d+2s}},
\end{equation*}
where $E$ is a Lebesgue measurable set and define the corresponding weighted Lebesgue space by
\begin{align*}
    L^1(\R^d;\mu) := \bigg\{ u \in L^1_{\loc}(\R^d) : \Vert u \Vert_{L^1(\R^d;\mu)} \equiv \int_{\R^d} \frac{|u(x)|}{(1+|x|)^{d+2s}}\d x < \infty \bigg\}.
\end{align*}

We introduce a bilinear form that is pertinent to the kernel $K$:
\begin{align*}
\cE^{(t)}(u,v) := \int_{\R^d} \int_{\R^d} (u(x) - u(y))(v(x)-v(y))K(t;x,y) \d y \d x.
\end{align*}

We give three definitions of solutions, a local one, a global one and one for Cauchy problems.

\begin{definition}[local solution]\label{Def:local-sol}
Let $I \subset \R$ be an interval and $\Omega \subset \R^d$ be a bounded domain.
We say that 
$$
u \in L^{\infty}(I;L^2_{\loc}(\Omega)) \cap L^2(I;H^s_{\loc}(\Omega)) \cap L^1(I;L^1(\R^d;\mu))
$$ is a weak supersolution to
\begin{align}
\label{eq:PDE-domains}
    \partial_t u - \mathcal{L}_t u = 0 ~~ \text{ in } I \times \Omega,
\end{align}
if for any $\phi \in H^1(I;L^2(\Om)) \cap L^2(I;H^s(\R^d))$ with $\phi \ge 0$ and $\supp(\phi) \Subset I \times \Omega$ it holds
\begin{align*}
        -\int_I \int_{\R^d} u(t,x) \partial_t \phi(t,x) \d x \d t + \int_I \cE^{(t)}(u(t),\phi(t)) \d t \ge 0.
    \end{align*}
    We say that $u$ is a weak subsolution to \eqref{eq:PDE-domains} if $-u$ is a weak supersolution, whereas $u$ is a weak solution to \eqref{eq:PDE-domains} if it is a weak supersolution and a weak subsolution.
\end{definition}

\begin{definition}[global solution]\label{Def:global-sol}
    Let $I\subset \R$ be an interval. We say that
    $$
    u \in L^{\infty}(I;L^2_{\loc}(\R^d)) \cap L^2(I;H^s_{\loc}(\R^d)) \cap L^1(I;L^1(\R^d;\mu))
    $$ is a global weak solution to 
    \begin{align*}
    \partial_t u - \mathcal{L}_t u = 0 ~~ \text{ in } I \times \R^d,
\end{align*}
if it is a weak solution in $I\times B_R$ for any $R>0$ in the sense of \autoref{Def:local-sol}.
\end{definition}

\begin{remark}
    It should be stressed that our definition of global solutions does not require $u(t,\cdot)$ to belong to $ H^s(\R^d)$ or $L^2(\R^d)$.
\end{remark}

\begin{remark}
    Note that time derivatives of solutions are not assumed to exist. However, by a standard time mollification argument (see \cite{FeKa13}, \cite{KaWe24a}, \cite{Lia24b}) we can work with a slightly stronger supersolution concept. That is, for a.e. $t \in I$:
    \begin{align*}
        \int_{\R^d} \partial_t u(t,x) \varphi(x) \d x + \cE^{(t)}(u(t),\varphi) \ge 0
    \end{align*}
    for any $\varphi \in H^s(\R^d)$ with $\supp(\varphi) \subset \Omega$ and $\varphi \ge 0$. Here, $\partial_t u$ denotes the $L^2$-derivative of $u$. The same remark holds for subsolutions.
\end{remark}

Solutions to the nonlocal Cauchy problem are defined as follows.  

\begin{definition}[Cauchy problem]
\label{def:Cauchy}
    Let $T > 0$ and $f \in L^2(\R^d)$. We say that $$
    u\in L^{\infty}((0,T);L^2(\R^d))\cap L^2((0, T); H^s(\R^d))
    $$
    is a weak solution to the Cauchy problem
    \begin{align*}
        \begin{cases}
            \partial_t u - \mathcal{L}_t u &= 0 ~~ \text{ in } (0,T) \times \R^d,\\
            u(0) &= f ~~ \text{ in } \R^d,
        \end{cases}
    \end{align*}
    if $u$ is a global solution in the sense of \autoref{Def:global-sol}, and $u(t) \to f$ in $L^2_{\loc}(\R^d)$ as $t \searrow 0$.
\end{definition}

\begin{remark}\label{Rmk:dual-prob}
The concept of dual Cauchy problem is analogous:
    \begin{align*}
        \begin{cases}
            \partial_t v + \mathcal{L}_t v &= 0 ~~ \text{ in } (0,T) \times \R^d,\\
            v(T) &= f ~~ \text{ in } \R^d.
        \end{cases}
    \end{align*}
    Let $u$ be a solution to the Cauchy problem for $\mathcal{L}_{T-t}$ in $(0,T) \times \R^d$ with $u(0) = f$ in the sense of \autoref{def:Cauchy}. Then, it holds that $v(t,x)= u(T-t,x)$ is a solution to the dual problem. Therefore, all interior regularity results for $u$ also hold true for $v$.
\end{remark}

\section{Collecting tools}
\label{sec:tools}
In this section, we collect some tools most of which are known to experts. They include some algebraic estimates, an integral estimate, boundedness estimates, weak Harnack estimates and continuity estimates of solutions.

\subsection{Algebraic estimates}

The goal of this section is to prove the following lemma.

\begin{lemma}
\label{lemma:algebraic-est}
Let $u(\cdot)> 0$ and $\eta(\cdot) \ge 0$, whereas $\eps \in (0,1)$. There exist $c_1,c_2 > 0$, depending only on $\eps$, such that for any $x,y \in \R^d$:
\begin{align*}
(u(y) - u(x))(\eta^2 u^{-\eps}(x) - \eta^2 u^{-\eps}(y)) &\ge c_1 (\eta u(x) - \eta u(y))^2 (u^{-\eps-1}(x) \wedge u^{-\eps-1}(y)) \\
&\quad - c_2 (u^{-\eps+1}(x) \vee u^{-\eps+1}(y))(\eta(x) - \eta(y))^2.
\end{align*} 
\end{lemma}

In order to prove \autoref{lemma:algebraic-est}, we first recall the following lemma (see \cite[Lemma 3.2]{KaWe22}).

\begin{lemma}
For any $\eta_1,\eta_2 \ge 0$ and $t,s > 0$:
	\begin{align}
		\label{eq:genaux4}
		(\eta_1^2 \wedge \eta_2^2)  |t - s|^2 &\ge \frac{1}{2} |\eta_1 t - \eta_2 s|^2  - (\eta_1 - \eta_2)^2 ( t^2 \vee s^2 ),\\
		\label{eq:genaux5}
		(\eta_1^2 \vee \eta_2^2) |t-s|^2 &\le 2 |\eta_1 t - \eta_2 s|^2 +2 (\eta_1 - \eta_2)^2 (t^2 \vee s^2 ).
	\end{align}
\end{lemma}

Moreover,  the next lemma will be applied with $g(t) = t^{-\eps}$.

\begin{lemma}
\label{lemma:algebra-Moser-testing}
Let $g: (0,\infty) \to (0,\infty)$ be decreasing. For any $\eta_1,\eta_2 \ge 0$ and $t,s > 0$, we have
\begin{align*}
(t-s)(\eta_1^2 g(s) - \eta_2^2 g(t)) &\ge (t-s)(g(s) - g(t)) ( \eta_1^2 \wedge \eta_2^2 ) \\
&\quad + (t-s) (g(s) \wedge g(t) )(\eta_1^2 - \eta_2^2).
\end{align*}
\end{lemma}

\begin{proof}
To begin with, we observe the following algebraic identities:
\begin{align*}
(\eta_1^2 g(s) - \eta_2^2 g(t)) &= (g(s) - g(t))\eta_1^2 + g(t)(\eta_1^2 - \eta_2^2) \\
&= (g(s) - g(t))\eta_2^2 + g(s)(\eta_1^2 - \eta_2^2).
\end{align*}
Since $g$ is decreasing, it must hold that $(t - s)(g(s) - g(t)) \ge 0$. Thus, we can estimate
\begin{align*}
(t - s)& (\eta_1^2 g(s) - \eta_2^2 g(t)) \\
&= 
\begin{cases}
(t - s) (g(s) - g(t)) \eta_1^2 + (t - s) g(t)(\eta_1^2 - \eta_2^2), ~~ \text{ if } g(t) \le g(s),\\
(t - s) (g(s) - g(t)) \eta_2^2 + (t - s) g(s)(\eta_1^2 - \eta_2^2), ~~ \text{ if } g(s) \le g(t),\\
\end{cases}\\
&\ge (t - s) (g(s) - g(t)) ( \eta_1^2 \wedge \eta_2^2 ) + (t-s) ( g(s) \wedge g(t) ) (\eta_1^2 - \eta_2^2).
\end{align*}
This is the desired inequality. 
\end{proof}

Let us now prove \autoref{lemma:algebraic-est}:

\begin{proof}[Proof of \autoref{lemma:algebraic-est}]
The proof relies on the following auxiliary observations.

First, it holds by the mean value theorem for any distinct $t,s > 0$ that
\begin{align}
\label{eq:algebraic-est-1}
(t-s)(s^{-\eps} - t^{-\eps}) = - (t-s)^2 \frac{s^{-\eps} - t^{-\eps}}{s-t} \ge \eps (t-s)^2 (s^{-\eps-1} \wedge t^{-\eps-1}).
\end{align}
Second, for any $t,s > 0$ we have
\begin{align}
\label{eq:algebraic-est-2}
t^{-\eps} \wedge s^{-\eps} = (s^{\frac{-\eps-1}{2}} \wedge t^{\frac{-\eps-1}{2}}) (s^{\frac{-\eps+1}{2}} \vee t^{\frac{-\eps+1}{2}}).
\end{align}
This can be seen easily by checking the cases $t > s$ and $s > t$ separately. Similarly,
\begin{align}
\label{eq:algebraic-est-3}
(t^{-\eps-1} \wedge s^{-\eps-1})(t^2 \vee s^2) = (t^{-\eps+1} \vee s^{-\eps + 1}).
\end{align}

Now, by \autoref{lemma:algebra-Moser-testing}, we have
\begin{align*}
(u(y) - u(x))(\eta^2 u^{-\eps}(x) - \eta^2 u^{-\eps}(y)) &\ge (u(y) - u(x))(u^{-\eps}(x) - u^{-\eps}(y))(\eta^2(x) \wedge \eta^2(y)) \\
&\quad + (u(y) - u(x))(u^{-\eps}(x) \wedge u^{-\eps}(y))(\eta^2(x) - \eta^2(y)) \\
&= I_1 + I_2.
\end{align*}
For $I_1$, we apply \eqref{eq:algebraic-est-1}, \eqref{eq:genaux4}, and \eqref{eq:algebraic-est-3} to deduce
\begin{align*}
I_1 &\ge \eps (u(x) - u(y))^2 (u^{-\eps-1}(x) \wedge u^{-\eps-1}(y)) (\eta^2(x) \wedge \eta^2(y)) \\
&\ge \frac{\eps}{2} (\eta u(x) - \eta u(y))^2 (u^{-\eps-1}(x) \wedge u^{-\eps-1}(y)) - \eps(\eta(x) - \eta(y))^2 (u^{-\eps+1}(x) \vee u^{-\eps+1}(y)).
\end{align*}
For $I_2$, we apply \eqref{eq:algebraic-est-2}, \eqref{eq:genaux5}, and \eqref{eq:algebraic-est-3} to deduce that for any $\delta > 0$ there exists $c(\delta) > 0$ such that
\begin{align*}
I_2 &= (u(y) - u(x))(\eta(x) + \eta(y)) (u^{\frac{-\eps-1}{2}}(x) \wedge u^{\frac{-\eps-1}{2}}(y))  (\eta(x) - \eta(y))  (u^{\frac{-\eps+1}{2}}(x) \vee u^{\frac{-\eps+1}{2}}(y)) \\
&\ge -\delta (u(x) - u(y))^2 (\eta^2(x) \vee \eta^2(y)) (u^{-\eps-1}(x) \wedge u^{-\eps-1}(y)) \\
&\quad - c(\delta) (\eta(x) - \eta(y))^2 (u^{-\eps+1}(x) \vee u^{-\eps+1}(y)) \\
&\ge - 2\delta (\eta u(x) - \eta u(y))^2 (u^{-\eps-1}(x) \wedge u^{-\eps-1}(y)) \\
&\quad -(2\delta + c(\delta)) (\eta(x) - \eta(y))^2 (u^{-\eps+1}(x) \vee u^{-\eps+1}(y)).
\end{align*}
Thus, the desired result follows upon choosing $\delta = \frac{\eps}{4}$, and combining the estimates for $I_1, I_2$.
\end{proof}

\subsection{An integral estimate}
The following integral estimate will be useful.
\begin{lemma}\label{Lm:tail-weight-estimate}
    There exists $c>1$ depending on $d$ and $s$, such that
\begin{align*}
    \int_{  B^c_1} \frac1{(1 + |x-y|)^{d+2s}} \frac{ \d y}{|y|^{d+2s}} \le \frac{c}{(1 + |x|)^{d+2s}}
\end{align*}
holds for all $x \in \R^d$.
\end{lemma}

\begin{proof}
Observe that for any $x \in \R^d$, we have
    \begin{align*}
           \int_{  B^c_1} \frac1{(1 + |x-y|)^{d+2s}} \frac{ \d y}{|y|^{d+2s}} \le\int_{  B^c_1}   \frac{ \d y}{|y|^{d+2s}} = c(d,s),
    \end{align*}
    whereas if $x \in  B^c_{2}$, we can refine the estimate. Indeed, observe that
    \begin{equation*}
    \left\{
        \begin{array}{cc}
            \displaystyle\frac1{(1+|x-y|)^{d+2s}} \le \frac{c(d)}{|x|^{d+2s}} &\text{if}\quad y \not \in B_{\frac{|x|}{4}}(x), \\ [5pt]
           \displaystyle \frac1{|y|^{d+2s}} \le \frac{c(d)}{|x|^{d+2s}} &\text{if}\quad y \in B_{\frac{|x|}{4}}(x).
        \end{array}\right.
    \end{equation*}
    As a result, we can estimate
    \begin{align*}
    \int_{  B_1^c} \frac1{(1 + |x-y|)^{d+2s}} \frac{\d y}{|y|^{d+2s}} 
        &  \le  \frac{c}{|x|^{d+2s}} \int_{ (B_1 \cup B_{\frac{|x|}{4}}(x))^c} \frac{\d y}{|y|^{d+2s}}  \\
        &\quad + \frac{c}{|x|^{d+2s} }\int_{B_{\frac{|x|}{4}}(x)} \frac{\d y}{(1 + |x-y|)^{d+2s}}  \\
        &\le \frac{c}{ |x|^{d+2s}} + \frac{c}{ |x|^{d+2s}} \int_{B_{\frac{|x|}{4}}} \frac{\d y}{(1 + |y|)^{d+2s}}  \\
        &\le \frac{c}{ |x|^{d+2s}}
    \end{align*}
for some $c=c(d,s)$.
Hence, we have shown the desired estimate.
\end{proof}

\subsection{Boundedness estimates}
The following various forms of boundedness estimates are consequences of local estimates for subsolutions appearing in \cite{KaWe24c} and \cite{Lia24}.

\begin{proposition}\label{Prop:bd}
    Assume that the kernel $K$ satisfies the upper bound in \eqref{eq:Kcomp} and \eqref{eq:coercive}. Let $u$ be a nonnegative, global solution in $(0,T]\times\R^d$ in the sense of \autoref{Def:global-sol}. Consider $x_o\in\R^d$ and $t_o\in(0,  \frac12 T)$.  Then, there exists $c>1$ depending only on $d$, $s$, $\lm$, $\Lm$, $T$ and $t_o$, such that
    \begin{align*}
        \sup_{(t_o,T]\times B_5(x_o)} u \le c\int_0^T\int_{B_{10}(x_o)} u(x,t)\d x \d t + c \int_0^T\int_{B_{10}^c(x_o)}\frac{u(x,t)}{|x-x_o|^{d+2s}}\d x \d t.
    \end{align*}
    Moreover, we can trace the dependence of $c$ as
    \[
    c=\widetilde{c}(d,s,\lm,\Lm)\bigg(\frac{T}{t_o}\bigg)^{d+2s} \vee \bigg[\bigg(\frac{T}{t_o}\bigg)^{q(d+2)+\frac12(d+1)}\bigg(\frac{10^{2s}}{T}+2^{d+2s}\bigg)^{2q}\bigg]
    \]
    where $2q=1+\frac{d}{2s}$.
\end{proposition}
\begin{proof}
According to \cite[Proposition~5.1]{Lia24b} with $\varep=1$, $\rho=10$, $\theta=T$ and $\sig=\frac{T-t_o}{T}$, we have
    \begin{align*}
\sup_{(t_o, T]\times B_{10\sig}(x_o)} u &\le \Big(\frac{4}{1-\sig}\Big)^{d+2s} \int_0^T\int_{B^c_{10\sig}(x_o)}\frac{u(t,x)}{|x-x_o|^{d+2s}} \d x\d t\\
&\qquad+\frac{\displaystyle c   \Big(\frac{10^{2s}}{T}+\frac1{\sig^{d+2s}}\Big)^{2q}  }{(1-\sig)^{ q(d+2) +\frac{d+1}{2}}}  \int_0^T\int_{B_{10}(x_o)} u(t,x)\,\d x\d t 
\end{align*}
for some $c=c(d,s,\lm,\Lm)$. Note that $\sig\in(\frac12, 1)$, and hence we may take supreme over $B_5(x_o)$ on the left-hand side, whereas the first integral on the right-hand side can be estimated by
\begin{align*}
     \int_0^T\int_{B^c_{5}(x_o)}&\frac{u(t,x)}{|x-x_o|^{d+2s}} \d x\d t \\
     &=  \int_0^T\int_{B_{10}(x_o)\setminus B_{5}(x_o)}\frac{u(t,x)}{|x-x_o|^{d+2s}} \d x\d t  + \int_0^T\int_{B^c_{10}(x_o)}\frac{u(t,x)}{|x-x_o|^{d+2s}} \d x\d t\\
     & \le \int_0^T\int_{B_{10}(x_o)} u(t,x) \d x\d t + \int_0^T\int_{B^c_{10}(x_o)}\frac{u(t,x)}{|x-x_o|^{d+2s}} \d x\d t.
\end{align*}
Combining these estimates, we conclude. 
\end{proof}

\begin{proposition}\label{Prop:bd:2}
   Assume that the kernel $K$ satisfies the upper bound in \eqref{eq:Kcomp} and \eqref{eq:coercive}.  Let $u$ be a nonnegative, global solution in $(0,T]\times\R^d$ in the sense of \autoref{Def:global-sol}. Consider $x_o\in\R^d$ and $t_o\in(0, T]$. Then, there exists $c>1$ depending only on the data $\{d, s,\lm,\Lm\}$, such that
    \begin{align*}
        \sup_{(\frac14 t_o, t_o]\times B_1(x_o)} u \le  \frac{c}{1\wedge t_o^{1+\frac{d}{2s}}} \int_0^{t_o}\int_{\R^d}\frac{u(x,t)}{(1+|x-x_o|)^{d+2s}}\d x \d t.
    \end{align*}
\end{proposition}
\begin{proof}
Let us view $u$ as a global solution in $(0,t_o]\times \R^d$ and apply \autoref{Prop:bd} with $(t_o,T)$ replaced by $(\frac14 t_o, t_o)$ to obtain
    \begin{align*}
\sup_{(\frac14 t_o, t_o]\times  B_{1}(x_o)} u  &\le c     \int_0^{t_o}\int_{B_{10}(x_o)} u(t,x)\,\d x\d t+ c \int_0^{t_o}\int_{B^c_{10}(x_o)}\frac{u(t,x)}{|x-x_o|^{d+2s}} \d x\d t 
\end{align*}
for some $c=c(d,s,\lm,\Lm, t_o)$ which can be traced as
\[
c=\frac{\widetilde{c}(d,s,\lm,\Lm)}{1\wedge t_o^{1+\frac{d}{2s}}}.
\]
The first integral on the right-hand side is estimated by
\[
  11^{d+2s} \int_0^{t_o}\int_{B_{10}(x_o)} \frac{  u(t,x)}{(1+|x-x_o|)^{d+2s}}\,\d x\d t,
\]
because $1+|x-x_o|\le 11$ when $x\in B_{10}(x_o)$,
whereas the second integral is  estimated by
\[
2^{d+2s} \int_0^{t_o}\int_{B^c_{10}(x_o)}\frac{u(t,x)}{(1+|x-x_o|)^{d+2s}} \d x\d t,
\]
because $|x-x_o|\ge\frac12 (1+|x-x_o|)$ when $x\in B^c_{10}(x_o)$.
Now, we can conclude by combining the two integrals.
\end{proof}

\subsection{Weak Harnack estimates}
The following positivity estimates are a consequence of the local estimates in \cite{KaWe24c}; see also \cite{Lia24}.

\begin{proposition}\label{Prop:WHI-global}
Assume that the kernel $K$ satisfies \eqref{eq:Kcomp}.
    Let $u$ be a nonnegative, global solution in $(0, T]\times\R^d$ in the sense of \autoref{Def:global-sol}.  Consider $x_o\in\R^d$ and $t_o\in(0, T]$. Then, there exists $c>1$ depending only on the data $\{d, s,\lm,\Lm\}$, such that 
    \begin{align*}
    c\,\inf_{(\frac34 t_o, t_o]\times B_{t_o^{1/2s}}(x_o)} u(t,\cdot)\ge \frac{1}{1\vee t_o^{1+\frac{d}{2s}} }\int_{\frac14 t_o}^{\frac12 t_o}\int_{\R^d}\frac{u(t,x)}{(1+|x-x_o|)^{d+2s}}\d x \d t.
    \end{align*}
\end{proposition}

\begin{proof}
    By \cite[Theorem~1.9]{KaWe24c}, there exists $c=c(d,s,\lm,\Lm)$, such that if for some $\rho>0$ and $t_o\in(0,T)$ it holds that
    \[
    (t_o-4\rho^{2s},t_o]\subset (0,T],
    \]
    then we have for any $t\in(t_o-\rho^{2s},t_o]$ that
    \begin{align*}
        c\, \inf_{B_{2^{1/s}\rho}(x_o)}u(t,\cdot) &\ge \bint_{t_o-3\rho^{2s}}^{t_o-2\rho^{2s}}\bint_{B_\rho(x_o)} u(t,x) \d x \d t 
        +\int_{t_o-3\rho^{2s}}^{t_o-2\rho^{2s}}\int_{B^c_\rho(x_o)} \frac{u(t,x)}{|x-x_o|^{d+2s}} \d x \d t\\
        &\ge \frac{1}{1\vee (\rho^{2s}|B_\rho|)}\int_{t_o-3\rho^{2s}}^{t_o-2\rho^{2s}}\int_{\R^d}\frac{u(t,x)}{(1+|x-x_o|)^{d+2s}}\d x \d t. 
    \end{align*}
 
    If we choose $\rho$ to satisfy $t_o=4\rho^{2s}$, then we obtain for any $t\in(\frac34 t_o, t_o]$ that
    \[
    c\,\inf_{B_{t_o^{1/2s}}(x_o)} u(t,\cdot)\ge \frac{1}{1\vee t_o^{1+\frac{d}{2s}} }\int_{\frac14 t_o}^{\frac12 t_o}\int_{\R^d}\frac{u(t,x)}{(1+|x-x_o|)^{d+2s}}\d x \d t.
    \]
    This finishes the proof.
\end{proof}

\subsection{Continuity of local solutions}
The following local continuity result, which we take from \cite[Theorem~1.1]{Lia24c}, is needed when we discuss the failure of time-insensitive Harnack estimates for local solutions. 
\begin{theorem}[modulus of continuity]\label{Thm:mod-con}
Assume that the kernel $K$ satisfies \eqref{eq:Kcomp}. Let $u$ be a locally bounded, local, weak solution to \eqref{Eq:1:1} in $I\times\Om$ in the sense of \autoref{Def:local-sol}.
	Then $u$ is locally continuous in $I\times\Om$. More precisely, consider the  cylinder  $(t_o-\widetilde{R}^{2s},t_o]\times B_{\widetilde{R}}(x_o)$
 included in $I\times \Om$ and a parameter $R\in(0,\widetilde{R})$.
	There exist constants $c>1$ and $\be\in(0,1)$ depending on the data $\{d, s, \lm, \Lm\}$, such that for any $r\in(0,R)$, we have
	\begin{equation*}
	\osc_{(t_o-r^{2s},t_o]\times B_{ r}(x_o)}u \le 2 \boldsymbol\om \Big(\frac{r}{R}\Big)^{\be}  + c  \int^{t_o}_{t_o-(rR)^{s}}    \int_{B^c_{\widetilde{R}}(x_o)} \frac{|u(t,x)|}{|x-x_o|^{d+2s}}\d x \d t,
	\end{equation*} 
where
 $$\boldsymbol\om=2\sup_{(t_o-\widetilde{R}^{2s},t_o]\times B_{\widetilde{ R}}(x_o)}|u| +\int^{t_o}_{t_o - \widetilde{R}^{2s}}    \int_{ B^c_{\widetilde{R}}(x_o)} \frac{|u(t,x)|}{|x-x_o|^{d+2s}}\d x \d t.$$
\end{theorem}

\begin{remark}
    Even though the above theorem is stated under condition \eqref{eq:Kcomp} on $K$, the proof in \cite{Lia24c} is robust and hinges upon an expansion of positivity which can also be produced under more general conditions; see~\autoref{lemma:expansion-of-positivity}.
\end{remark}

\section{Some local estimates for supersolutions}
\label{sec:local}

In this section we prove two additional local estimates for local supersolutions in the sense of \autoref{Def:local-sol}. One is a reverse H\"older inequality; the other is an energy decay estimate. They will play important roles in  proving both the time-insensitive Harnack estimates and the improved weak Harnack estimate.

\subsection{A reverse H\"older estimate}

Let us first recall the following lemma, which follows by Moser iteration for small positive exponents. For $q \le 1$, it is completely standard. In order to obtain $1 < q < 1 + \frac{2s}{d}$, one has to do one more iteration step.

\begin{proposition}\label{Lm:Lq-Lsig}
Assume that the kernel $K$ satisfies the upper bound in \eqref{eq:Kcomp} and \eqref{eq:coercive}.
Let $u\ge0$ be a weak supersolution in $(0,4) \times B_4$ in the sense of \autoref{Def:local-sol}. Then, for any $\sig$ and $q$ satisfying $0 < \sigma < q < 1 + \frac{2s}{d}$, it holds
\begin{align*}
\bigg(\int_{0}^{\frac12} \int_{B_{\frac12}} u^q(t,x) \d x \d t \bigg)^{\frac1q} \le c \bigg(\int_{0}^{1} \int_{B_1} u^{\sigma}(t,x) \d x \d t \bigg)^{\frac1\sigma},
\end{align*}
where $c > 1$ depends only on the data $\{d, s, \lambda, \Lambda\}$, $\sigma$ and $q$.
\end{proposition}

\begin{proof}
First, by \cite[Theorem 4.2]{KaWe22}, there exists $c > 0$, depending only on $d $, $ s$, $\lambda$ and $\Lambda$, such that
\begin{align}
\label{eq:Moser-small-pos}
\int_{0}^{\frac{3}{4}}\int_{B_{\frac34}} u(t,x) \d x \d t \le c \left( \int_0^1 \int_{B_1} u^{\sigma}(t,x) \d x \d t \right)^{\frac1\sigma}.
\end{align}
Moreover, note that it holds for any $p \in (0,1)$
\begin{align*}
\Vert u^p \Vert_{L^{1+\frac{2s}{d}}((0,\frac{1}{2}) \times B_{\frac{1}{2}})} \le c\, \Vert u^p \Vert_{L^1((0,\frac{3}{4}) \times B_{\frac{3}{4}})},
\end{align*}
where $c > 0$ depends on $d$, $s$, $\lambda$, $\Lambda$, and $p$. 
This estimate is proved in \cite[Proposition 3.6]{FeKa13} (see also \cite[(4.9)]{KaWe22}) for $p \in (p^{\ast},\frac{d}{d+2s})$ for some $p^{\ast} \in (0,1)$ with a constant $c > 0$ that is independent of $p$. However, as it is mentioned in \cite[Remark after Proposition 3.6]{FeKa13}, it continues to hold true for arbitrary $p \in (0,1)$ with exactly the same proof, if the constant $c > 0$ is allowed to depend on $p$. As a consequence, we have
\begin{align*}
\bigg(\int_0^{\frac{1}{2}} \int_{B_{\frac12}} u^{p (1 + \frac{2s}{d} )} (t,x) \d x \d t \bigg)^{\frac1{p(1 + \frac{2s}{d})}} &\le c \bigg( \int_0^{\frac{3}{4}} \int_{B_{\frac34}} u^p (t,x) \d x \d t \bigg)^{\frac1p} \\
&\le c \int_{0}^{\frac{3}{4}}\int_{B_{\frac34}} u(t,x) \d x \d t \\
&\le c \bigg( \int_0^1 \int_{B_1} u^{\sigma}(t,x) \d x \d t \bigg)^{\frac1\sigma},
\end{align*}
where we also used Jensen's inequality and \eqref{eq:Moser-small-pos}.  Then, by choosing $p = q /(1 + \frac{2s}{d}) \in (0,1)$ we conclude the proof.
\end{proof}

\subsection{An energy decay estimate}
As mentioned in the introduction, the following energy decay estimate plays an important role in our theory and will be used repeatedly.
\begin{proposition}
\label{lemma:energy-decay}
Assume that the kernel $K$ satisfies the upper bound in \eqref{eq:Kcomp} and \eqref{eq:coercive}. Let $u$ be a nonnegative weak supersolution in $(0,4)\times B_4$ in the sense of \autoref{Def:local-sol} satisfying
\begin{align*}
\int_0^{2} \int_{B_2} u^q \d x \d t <\infty
\end{align*}
for some $q \in (1,1 + \frac{2s}{d})$. Then, there exists $c > 1$ depending on the data $\{d, s, \lambda, \Lambda\}$ and $q$, such that for any $\tau \in (0,1)$ it holds
\begin{align*}
\int_0^{\tau} \int_{B_{1}} \int_{B_{1}} |u(t,x) - u(t,y)| K(t;x,y)|x-y| \d y \d x \d t \le c \,\tau^{1-\frac{1}{q}} \bigg( \int_0^{2\tau} \int_{B_{2}} u^{q} \d x \d t \bigg)^{\frac{1}{q}}.
\end{align*}
\end{proposition}

Prior to the proof of \autoref{lemma:energy-decay}, we need the following energy estimate.
\subsubsection{An energy estimate for local supersolutions}
\begin{lemma}\label{Lm:energy-eps}
   Assume that the kernel $K$ satisfies the upper bound in \eqref{eq:Kcomp} and \eqref{eq:coercive}. Let $u$ be a nonnegative weak supersolution in $I\times\Om$ in the sense of \autoref{Def:local-sol}. Suppose $(t_o,t_1)\times B_R\subset I\times\Om$ and $\varep\in (0,1)$. Then, there exists $c>1$ depending only on $\varep$, such that for every nonnegative, piecewise smooth cutoff function $\eta$ satisfying $\supp \eta(t,\cdot)\subset B_R$, $\forall\, t\in(t_o,t_1)$, we have
    \begin{align*}
        &\int_{B_R} \eta^2 u^{-\varep+1}(t_o,x) \d x-\int_{B_R} \eta^2 u^{-\varep+1}(t_1,x)\,\d x\\
        &\qquad+\int_{t_o}^{t_1} \int_{B_R} \int_{B_R} (\eta u(t,x) - \eta u(t,y))^2 (u^{-\eps-1}(t,x) \wedge u^{-\eps-1}(t,y)) K(t;x,y) \d y \d x \d t\\
        &\le c \int_{t_o}^{t_1} \int_{B_R} |\pl_t\eta^2| u^{-\eps+1}(t,x) \d x \d t \\
&\qquad + c \int_{t_o}^{t_1} \int_{B_{R}} \int_{\R^d \setminus B_R}  \eta^2 \tilde{u}^{-\eps+1}(t,x) K(t;x,y) \d y \d x \d t \\
& \qquad + c \int_{t_o}^{t_1} \int_{B_R} \int_{B_R} (\tilde{u}^{-\eps+1}(t,x) \vee \tilde{u}^{-\eps+1}(t,y))(\eta(t,x) - \eta(t,y))^2 K(t;x,y) \d y \d x \d t.
    \end{align*}
\end{lemma}

\begin{proof}
Let $\tilde{u}:=u+\dl$ for some $\dl\in(0,1)$. Observe that $\tilde{u}$ satisfies the same equation as $u$.  We use $\phi = -\eta^2\tilde{u}^{-\eps} $ as a testing function in the weak formulation of $\tilde{u}$, modulo a time mollification, and arrive at
\begin{align*}
-\int_{t_o}^{t_1} \int_{B_R} \partial_t \tilde{u}(t,x) \eta^2 \tilde{u}^{-\eps}(t,x) \d x \d t + \int_{t_o}^{t_1} \cE^{(t)}(\tilde{u}(t),-\eta^2 \tilde{u}^{-\eps}(t)) \d t \le 0.
\end{align*}
For the first term, using that $  \partial_t \tilde{u} \cdot \eta^2 \tilde{u}^{-\eps} = \frac{1}{-\eps+1}\eta^2 \partial_t \tilde{u}^{-\eps+1} $, we perform integration by parts in time. Moreover, the second term, by the definition of $\mathcal{E}^{(t)}$, can be split into a local term and a nonlocal term. As a result, we obtain
\begin{align*}
\frac1{1-\varepsilon}&\int_{B_R}  \eta^2\tilde{u}^{-\varep+1}(t_o,x)\,\d x-\frac1{1-\varepsilon}\int_{B_R} \eta^2 \tilde{u}^{-\varep+1}(t_1,x)\,\d x\\
    &\quad+\int_{t_o}^{t_1}\int_{B_R} \int_{B_R} (\tilde{u}(t,y) - \tilde{u}(t,x))(\eta^2 \tilde{u}^{-\eps}(t,x) - \eta^2 \tilde{u}^{-\eps}(t,y)) K(t;x,y) \d y \d x \d t\\
    &\le -\frac1{1-\varepsilon}\int_{t_o}^{t_1} \int_{B_R} \tilde{u}^{-\eps+1}(t,x) \partial_t \eta^2(t,x) \d x \d t\\
    &\qquad-2\int_{t_o}^{t_1} \int_{B_{R}} \int_{\R^d \setminus B_R} (\tilde{u}(t,y) - \tilde{u}(t,x))\eta^2 \tilde{u}^{-\eps}(t,x) K(t;x,y) \d y \d x \d t .
\end{align*}
The last term can be estimate from above by
\[
2\int_{t_o}^{t_1} \int_{B_{R}} \int_{\R^d \setminus B_R}  \eta^2 \tilde{u}^{-\eps+1}(t,x) K(t;x,y) \d y \d x \d t
\]
after discarding the negative contribution of $\tilde{u}(t,y)$.
The triple integral on the left can be estimated from below by evoking \autoref{lemma:algebraic-est} for $(x,y) \in B_R \times B_R$. This yields
\begin{align*}
    \int_{t_o}^{t_1}&\int_{B_R} \int_{B_R} (\tilde{u}(t,y) - \tilde{u}(t,x))(\eta^2 \tilde{u}^{-\eps}(t,x) - \eta^2 \tilde{u}^{-\eps}(t,y)) K(t;x,y) \d y \d x \d t\\
    &\ge c_1 \int_{t_o}^{t_1} \int_{B_R} \int_{B_R} (\eta \tilde{u}(t,x) - \eta \tilde{u}(t,y))^2 (\tilde{u}^{-\eps-1}(t,x) \wedge \tilde{u}^{-\eps-1}(t,y)) K(t;x,y) \d y \d x \d t \\
    &\quad-c_2\int_{t_o}^{t_1} \int_{B_R} \int_{B_R} (\tilde{u}^{-\eps+1}(t,x) \vee \tilde{u}^{-\eps+1}(t,y))(\eta(t,x) - \eta(t,y))^2 K(t;x,y) \d y \d x \d t .
\end{align*}
Collecting these estimates and letting $\dl\to0$ we finish the proof.
\end{proof}

\subsubsection{Proof of \autoref{lemma:energy-decay}}
Let $\delta \in (0,1)$ and $\eps := \frac{q-1}{2}$, and denote $\tilde{u} = u+\delta$. First, by H\"older's inequality we estimate
\begin{align*}
\int_0^{\tau} & \int_{B_{1}} \int_{B_{1}} |u(t,x) - u(t,y)| K(t;x,y)|x-y| \d y \d x \d t \\
& = \int_0^{\tau} \int_{B_{1}} \int_{B_{1}} |\tilde{u}(t,x) - \tilde{u}(t,y)| K(t;x,y)|x-y| \d y \d x \d t \\
& \le \bigg(\int_0^{\tau} \int_{B_{1}} \int_{B_{1}} (\tilde{u}^{\eps+1}(t,x) \vee \tilde{u}^{\eps+1}(t,y)) K(t;x,y)|x-y|^{2} \d y \d x \d t \bigg)^{1/2}  \\
&\quad \times \bigg( \int_0^{\tau} \int_{B_{1}} \int_{B_{1}} |\tilde{u}(t,x) - \tilde{u}(t,y)|^2 (\tilde{u}^{-\eps-1}(t,x) \wedge \tilde{u}^{-\eps-1}(t,y)) K(t;x,y) \d y \d x \d t \bigg)^{1/2} \\
&=: I_1 \times I_2.
\end{align*}
For $I_1$, we compute by the upper bound on $K$ in \eqref{eq:Kcomp} and again by H\"older's inequality
\begin{align*}
I_1 &\le 2 \bigg( \int_0^{\tau} \int_{B_{1}} \tilde{u}^{\eps+1}(t,x) \bigg( \int_{B_{1}} K(t;x,y) |x-y|^{2} \d y \bigg) \d x \d t \bigg)^{1/2} \\
&\le 2 \bigg( \int_0^{\tau} \int_{B_{1}} \tilde{u}^{\eps+1}(t,x) \bigg( \int_{B_{2}(x)} |x-y|^{-d-2s+2} \d y \bigg) \d x \d t \bigg)^{1/2} \\
&\le c \bigg( \int_0^{\tau} \int_{B_{1}} \tilde{u}^{\eps+1}(t,x) \d x \d t \bigg)^{1/2}\\
&\le c \,\tau^{\frac{\eps}{2(2\eps+1)}}  \bigg( \int_0^{\tau} \int_{B_{1}} \tilde{u}^{2\eps+1}(t,x) \d x \d t \bigg)^{\frac{1+\eps}{2(2\eps+1)}}.
\end{align*}
In order to estimate $I_2$, we evoke \autoref{Lm:energy-eps} with $(t_o,t_1)\times B_R\equiv (0,2\tau)\times B_2$. The cutoff function $\eta$ is chosen to be supported in $[0,2\tau)\times B_{3/2}$ satisfying $\eta \equiv 1$ in $[0,\tau)\times B_1$, $|\partial_t \eta|\le 2/\tau$, and $|\nabla\eta|\le 2$. 
Consequently, using also that $u \ge 0$, we estimate
\begin{align*}
I_2^2 
&\le - \int_{B_2} \eta^2 \tilde{u}^{-\varep+1}(0,x) \d x + c \int_0^{2\tau} \int_{B_2} |\partial_t \eta^2| \tilde{u}^{-\eps+1}(t,x)  \d x \d t  \\
&\quad + c \int_0^{2\tau} \int_{B_{2}} \int_{  B^c_2} \eta^2 \tilde{u}^{-\eps+1}(t,x) K(t;x,y) \d y \d x \d t \\
&\quad + c \int_0^{2\tau} \int_{B_2} \int_{B_2} (\tilde{u}^{-\eps+1}(t,x) \vee \tilde{u}^{-\eps+1}(t,y))(\eta(t,x) - \eta(t,y))^2 K(t;x,y) \d y \d x \d t  \\
& \le \frac{c}{\tau} \int_0^{2\tau} \int_{B_2} \tilde{u}^{-\eps+1}(t,x) \d x \d t \\
&\quad + c \int_0^{2\tau} \int_{B_{3/2}} \eta^2 \tilde{u}^{-\eps+1}(t,x) \bigg( \int_{  B^c_2} K(t;x,y) \d y \bigg) \d x \d t \\
& \quad + c \int_0^{2\tau} \int_{B_{2}} \tilde{u}^{-\eps+1}(t,x) \bigg( \int_{ B_2} (\eta(t,x) - \eta(t,y))^2 K(t;x,y) \d y \bigg) \d x \d t \\
&\le \frac{c}{\tau} \int_0^{2\tau} \int_{B_2} \tilde{u}^{-\eps+1}(t,x) \d x \d t.
\end{align*}
We collect these estimates and apply H\"older's inequality again to obtain
\begin{align*}
\int_0^{\tau} & \int_{B_{1}} \int_{B_{1}} |u(t,x) - u(t,y)| K(t;x,y) |x-y| \d y \d x \d t\le I_1\times I_2 \\
&\le c\, \tau^{\frac{\eps}{2(2\eps+1)}}  \bigg( \int_0^{\tau} \int_{B_{1}} \tilde{u}^{2\eps+1}(t,x) \d x \d t \bigg)^{\frac{1+\eps}{2(2\eps+1)}} \bigg( \frac1\tau\int_0^{2\tau} \int_{B_2} \tilde{u}^{-\eps+1}(t,x) \d x \d t \bigg)^{\frac12} \\
&\le c\, \tau^{\frac{2\eps}{2\eps+1}} \bigg( \int_0^{2\tau} \int_{B_{2}} \tilde{u}^{2\eps+1}(t,x) \d x \d t \bigg)^{\frac{1}{2\eps+1}}.
\end{align*}
 Clearly, upon taking the limit $\delta \searrow 0$ and recalling $2\varepsilon+1=q$, we finish the proof.

\section{An auxiliary function and its properties}
\label{sec:Psi}

This section is devoted to an auxiliary function generated by solving the dual problem (see \autoref{Rmk:dual-prob}) with a particularly chosen initial datum. This function will be used in the proof of the time-insensitive Harnack estimates.
\begin{lemma}
\label{lemma:Psi-bounds}
Assume that the kernel $K$ satisfies the upper bound in \eqref{eq:Kcomp} and \eqref{eq:coercive}.
    Let $\Psi$ be the solution to the dual problem
\begin{align*}
    \begin{cases}
        \partial_t \Psi + \mathcal{L}_t \Psi &= 0 ~~ \text{ in } (0,T) \times \R^d,\\ 
        \displaystyle\Psi(T,\cdot) &= (1 + |\cdot|)^{-d-2s} ~~ \text{ in } \R^d.
    \end{cases} 
\end{align*}
Then, we have for any $t \in (0,T)$ and $x \in \R^d$ that
\begin{align}
\label{eq:Psi-bounds}
    \frac{c^{-1}}{(1 + |x|)^{d+2s}} \le \Psi(t,x) \le \frac{c}{(1 + |x|)^{d+2s}},
\end{align}
where $c > 1$ depends only on $d$, $s$, $\lambda$, $\Lambda$, and $T$.
\end{lemma}

\begin{proof}

Let us define $\Phi(t) = \Psi(T-t)$ and observe that $\Phi$ is a solution to \begin{align}
\label{eq:Phi-problem}
\begin{cases}
    \partial_t \Phi - \tilde{\mathcal{L}}_t \Phi &= 0 ~~ \text{ in } (0,T) \times \R^d, \\
    \Phi(0) &= (1 + |\cdot|)^{-d-2s} ~~ \text{ in } \R^d,
\end{cases}
\end{align}
where $\tilde{\mathcal{L}}_t = \mathcal{L}_{T-t}$.  By the existence of the heat kernel $p_{t}(x,y)\equiv p_{0,t}(x,y)$ associated to $\tilde{\mathcal{L}}_t$ (see Appendix \ref{sec:appendix}), and since $\Phi(0) \in L^2(\R^d)$, we can represent $\Phi$ as 
\begin{align*}
        \Phi(t,x) = \int_{\R^d} \frac{p_{t}(x,y)}{(1 + |y|)^{d+2s}}  \d y.
\end{align*}

Moreover, we have the following heat kernel upper bound (see \cite{KaWe23})
\begin{align}
\label{eq:heat-kernel-upper}
    p_t(x,y) \le C \left( \frac1{t^{\frac{d}{2s}}} \wedge \frac{t}{|x-y|^{d+2s}} \right),
\end{align}
where $C > 0$ depends only on $d$, $s$, $\lambda$, and $\Lambda$.

Let us show the upper bound of $\Psi$ in \eqref{eq:Psi-bounds}. First, by \autoref{Lm:tail-weight-estimate} we have
\begin{align}
\label{eq:tail-weight-estimate}
    \int_{  B^c_1} \frac1{(1 + |x-y|)^{d+2s}} \frac{ \d y}{|y|^{d+2s}} \le \frac{c(d,s)}{(1 + |x|)^{d+2s}} \qquad \forall \, x \in \R^d.
\end{align}
 Then, we use \eqref{eq:heat-kernel-upper} and split the integral into two parts and estimate
\begin{align*}
    \Psi(t,x) &\le C \int_{ \{|x-y|>(T-t)^{\frac1{2s}}\} } \frac{T-t}{|x-y|^{d+2s}} \frac{\d y}{(1 + |y|)^{d+2s}}\\
    &\qquad +  C \int_{ \{|x-y|\le (T-t)^{\frac1{2s}}\} } \frac1{(T-t)^{\frac{d}{2s}}} \frac{\d y}{(1 + |y|)^{d+2s}} \\
    &=: I_1 +  I_2,
\end{align*}
where $C=C(d,s,\lm, \Lm)$. For an estimate of $I_2$, we discuss two cases. In the case $(T-t)^{\frac{1}{2s}} \le \frac{|x|}{2}$ we have $(1 + |x| - (T-t)^{\frac{1}{2s}}) \ge \frac{1}{2} (1 + |x|)$, and hence
\begin{align*}
    I_2\le C \frac{1}{(T-t)^{\frac{d}{2s}}}\int_{ \{|x-y|\le (T-t)^{\frac1{2s}}\} } \frac{\d y}{(1 + |x| - (T-t)^{\frac{1}{2s}})^{d+2s}} \le \frac{c(d,s,\lambda,\Lambda)}{(1 + |x|)^{d+2s}};
\end{align*}
in the case $(T -t)^{\frac{1}{2s}} \ge \frac{|x|}{2}$, we have
\begin{align*}
I_2 \le  C \frac{1}{(T-t)^{\frac{d}{2s}}}\int_{ \{|x-y|\le (T-t)^{\frac1{2s}}\} } \d y \le c(d,s,\lambda,\Lambda) \le \frac{c(d,s,\lambda,\Lambda,T)}{(1+|x|)^{d+2s}}.
\end{align*} 

For $I_1$, we consider several cases. First, when $T-t \ge 1$, we can use \eqref{eq:tail-weight-estimate} to deduce
\begin{align*}
    I_1 &= C (T-t) \int_{\{|y|>(T-t)^{\frac1{2s}}\} }  \frac1{(1 + |x-y|)^{d+2s}} \frac{\d y}{|y|^{d+2s}} \\
    &\le  CT \int_{ B^c_{1}} \frac1{(1 + |x-y|)^{d+2s}} \frac{\d y}{|y|^{d+2s}}\\
    &\overset{\eqref{eq:tail-weight-estimate}}{\le} \frac{c}{ (1 + |x|)^{d+2s}}
\end{align*}
for some $c=c(d,s, \lm,\Lm, T)$. 
Moreover, when $T-t < 1$ and $|x| \ge 1$, we have by the change of variables $y \mapsto (T-t)^{-\frac{1}{2s}} y$ and an application of \eqref{eq:tail-weight-estimate}: 
\begin{align*}
    I_1 &= C (T-t) \int_{\{|y|>(T-t)^{\frac1{2s}}\} }  \frac1{(1 + |x-y|)^{d+2s}} \frac{\d y}{|y|^{d+2s}} \\
    &\le C (T-t) \int_{\{|y|>(T-t)^{\frac1{2s}}\}}  \frac1{((T-t)^{\frac{1}{2s}} + |x-y|)^{d+2s}} \frac{\d y}{|y|^{d+2s}} \\
    &= C (T-t)^{-\frac{d+2s}{2s}} \int_{ B^c_1}  \frac1{(1 + |(T-t)^{-\frac{1}{2s}}x-y|)^{d+2s} }\frac{\d y}{|y|^{d+2s}} \\
    &\overset{\eqref{eq:tail-weight-estimate}}{\le} c (T-t)^{-\frac{d+2s}{2s}} \frac1{(1 + (T-t)^{-\frac{1}{2s}}|x|)^{d+2s}}\\
    &=\frac{c}{((T-t)^{\frac{1}{2s}} + |x|)^{d+2s}}\\
    &\le \frac{c}{ |x|^{d+2s}}
\end{align*}
for some $c=c(d,s, \lm,\Lm)$.
When $T-t < 1$ and $|x| \le 1$, we compute
\begin{align*}
    I_1 &= C (T-t) \int_{\{|y|>(T-t)^{\frac1{2s}}\} }  \frac1{(1 + |x-y|)^{d+2s}} \frac{\d y}{|y|^{d+2s}} \\
    &\le C (T-t) \int_{\{|y|>(T-t)^{\frac1{2s}}\} } \frac1{|y|^{d+2s}}  \d y \le c(d,s, \lm,\Lm).
\end{align*}
Therefore, combining these estimates, we arrive at
\begin{align*}
    I_1 \le \frac{c}{(1 + |x|)^{d+2s}}
\end{align*}
for some $c=c(d,s, \lm,\Lm, T)$.
Hence, we have established also the upper bound in \eqref{eq:Psi-bounds}.

As for the lower bound in \eqref{eq:Psi-bounds}, it follows directly from local estimates. In fact, by \cite[Lemma 3.2]{Lia24c}, if $\Psi(0,\cdot)\ge k$ in some ball $B_{2}(x_o)$ for some $ k>0$, then there exists $\nu_o$ depending only on the data $\{d,s,\lm,\Lm\}$, such that $\Psi\ge\frac12 k$ in $[0,\nu_o]\times B_1(x_o)$. Recalling the assigned initial datum of $\Psi$, we can apply this property with $k=(3+|x_o|)^{-d-2s}$ and hence obtain that $u(t, x_o)\ge\frac12(3+|x_o|)^{-d-2s}$ for any $(t, x_o)\in[0,\nu_o]\times \R^d$. Then, the proof is concluded by repeating this argument finitely many times to propagate the estimate further in time.
\end{proof}

Another property of $\Psi$ is the following.
\begin{lemma} \label{Lm:u-Psi}
Assume that the kernel $K$ satisfies the upper bound in \eqref{eq:Kcomp} and \eqref{eq:coercive}.
Let $\Psi$ be as in \autoref{lemma:Psi-bounds} and let $t_o\in(0,T)$. Then, there exists $c>1$ depending on $d$, $s$, $\lm$, $\Lm$ and $T$, such that for any $x_o\in\R^d$ we have
    \begin{align*}
    \int_{t_o}^{T}  &\int_{B_1(x_o)} \int_{B_1(x_o)}  |\Psi(t,x) - \Psi(t,y)| K(t;x,y)|x-y| \d y \d x \d t\le \frac{c}{(1+ |x_o|)^{d+2s}}  .
\end{align*}
\end{lemma}

\begin{proof} 
    Set $\Phi(t,x):=\Psi(T-t, x)$ and recall that it is a solution to the equation \eqref{eq:Phi-problem}, which is governed by the operator $\tilde{\mathcal{L}}_t = \mathcal{L}_{T-t}$ with kernel $\tilde{K}(t;x,y) := K(T-t;x,y)$. Take $q=1+\frac{s}{d}$ and apply \autoref{lemma:energy-decay} to get
    \begin{align*}
\int_0^{t_o} \int_{B_{1}(x_o)} \int_{B_{1}(x_o)} &|\Phi(t,x) - \Phi(t,y)| \tilde{K}(t;x,y)|x-y|  \d y \d x \d t \\
&\le c\, t_o^{1-\frac{1}{q}} \bigg( \int_0^{2 t_o} \int_{B_{2}(x_o)} \Phi^{q}(t,x) \d x \d t \bigg)^{\frac{1}{q}}
\end{align*}
for some $c$ depending on $d$, $s$, $\lambda$, $\Lambda$ and $q$. Now, apply the upper bound from \autoref{lemma:Psi-bounds} (recalling the definition of $\Phi$) to estimate the right-hand side and obtain
    \begin{align*}
\int_0^{t_o} \int_{B_{1}(x_o)} \int_{B_{1}(x_o)} |\Phi(t,x) - \Phi(t,y)| \tilde{K}(t;x,y)|x-y|  \d y \d x \d t \le \frac{c\, t_o}{(1+ |x_o|)^{d+2s}}
\end{align*}
for some $c$ depending on $d$, $s$, $\lambda$, $\Lambda$ and $T$. 
The same estimate can be repeated in $(t_o,2t_o)$ and so on. Namely, for $j\in\N$, we have
  \begin{align*}
\int_{(j-1)t_o}^{jt_o} \int_{B_{1}(x_o)} \int_{B_{1}(x_o)} |\Phi(t,x) - \Phi(t,y)| \tilde{K}(t;x,y)|x-y|  \d y \d x \d t \le \frac{c\, t_o}{(1+ |x_o|)^{d+2s}}
\end{align*}
as long as $(j+1)t_o\le T$.
Without loss of generality, we may assume $T/t_o$ is an integer. Then, adding the above estimate from 
$1$ to $T/t_o -1$ we obtain
    \begin{align*}
\int_0^{T- t_o} \int_{B_{1}(x_o)} \int_{B_{1}(x_o)} |\Phi(t,x) - \Phi(t,y)| \tilde{K}(t;x,y)|x-y|  \d y \d x \d t \le \frac{c\, T}{(1+ |x_o|)^{d+2s}}.
\end{align*}
Finally, reverting to $\Psi$ we conclude the proof.
\end{proof}

\section{Time-insensitive Harnack estimates}
\label{sec:time-insensitive}

In this section we prove the  time-insensitive Harnack estimates as stated in \autoref{Thm:0} and \autoref{Thm:1}.
There are two main ingredients in our approach.
One is the parabolic nonlocal Harnack inequality for {\it local} solutions established in \cite{KaWe24c}. The {\it globality} is exploited in the second ingredient -- the following weighted $L^1$ estimate, which is of independent interest.

\subsection{Time-insensitive $L^1 (\R^d;\mu)$ estimate}

Throughout this section we assume that $K$ satisfies the upper bound in \eqref{eq:Kcomp} and \eqref{eq:coercive}.

The following result asserts the weighted norm of a solution is comparable for each time instant.

\begin{proposition}
\label{lemma:weighted-L1-time-insensitive}
Assume that the kernel $K$ satisfies the upper bound in \eqref{eq:Kcomp} and \eqref{eq:coercive}.
    Let $u$ be a nonnegative, global weak solution in $(0,T]\times \R^d$ in the sense of \autoref{Def:global-sol}. Consider $t,\tau \in (0,T)$. Then, there exists $c > 1$, depending only on $d$, $s$, $\lambda$, $\Lambda$, and $T$, such that
    \begin{align*}
       c^{-1} \Vert u(\tau) \Vert_{L^1(\R^d;\mu)} \le \Vert u(t) \Vert_{L^1(\R^d;\mu)} \le c\, \Vert u(\tau) \Vert_{L^1(\R^d;\mu)}.
    \end{align*}
\end{proposition}

\begin{remark}\label{Rmk:L1-weight}
This result resembles \cite[Lemma~4.2]{BSV17}. However, the proof in \cite{BSV17} relies on an explicit calculation of $(-\Dl)^s\Phi$, where $\Phi(x)$ decays like $|x|^{-d-2s}$, while, in general, $\mathcal{L}_t \Phi$ defined in \eqref{eq:op} could diverge (even if $\Phi \in C_c^{\infty}(\R^d)$) due to the lack of regularity of the kernel $K$. This subtlety was overlooked in \cite[Lemma 2.7]{Str19}.  
\end{remark}

\begin{proof}

Let $\Psi$ be the solution to the dual problem
\begin{align*}
    \begin{cases}
        \partial_t \Psi + \mathcal{L}_t \Psi &= 0 ~~ \text{ in } (0,T) \times \R^d,\\
        \Psi(T,\cdot) &= (1 + |\cdot|)^{-d-2s} ~~ \text{ in } \R^d.
    \end{cases} 
\end{align*}
Recall that the two-sided decay of $\Psi$ has been shown in \autoref{lemma:Psi-bounds}.

Next, let $\xi_R \in C_c^{1}(B_{2R})$ be such that $0 \le \xi_R \le 1$, $\xi_R \equiv 1$ in $B_R$, and $|\nabla \xi_R| \le 4 R^{-1}$. Moreover, we fix $t_o\in(0,T)$ and consider arbitrary $\tau_o,\,\tau_1$ satisfying $t_o \le \tau_o < \tau_1 < T$. We introduce a sequence of functions $\zeta_k\in C_c^{1}(0,T)$, $k\in\N$ such that
\begin{equation*}
\left\{
    \begin{array}{cc}
        \zeta_k \to \chi_{[\tau_o, \tau_1]},  \quad
         \partial_t \zeta_k \to \delta_{\tau_o} - \delta_{\tau_1}\quad \text{as}\>\> k\to\infty , \\ [5pt]
         \displaystyle\supp(\zeta_k) \subset \Big[\frac{t_o+\tau_o}{2}, \frac{T +\tau_1}{2}\Big].
    \end{array}\right.
\end{equation*}

Modulo a proper time mollification, we observe 
\begin{align*}
        -\int_0^{T} \partial_t \zeta_k(t) \int_{\R^d} \xi_R(x) u(t,x) \Psi(t,x) \d x \d t  &=   \int_0^{T} \zeta_k(t) \int_{\R^d} \xi_R(x) \partial_t u(t,x) \Psi(t,x) \d x \d t\\
    &\quad + \int_0^{T} \zeta_k(t) \int_{\R^d} \xi_R(x) \partial_t \Psi(t,x) u(t,x)  \d x \d t .
\end{align*}
To estimate the right-hand side, we test the weak formulation of $u$ with $\zeta_k\xi_R \Psi$ and the weak formulation of $\Psi$ with $\zeta_k\xi_R u$ and we obtain
\begin{align*}
    \int_0^{T} \zeta_k(t) \int_{\R^d} \xi_R(x) \partial_t u(t,x) \Psi(t,x) \d x \d t&=-\int_0^{T} \zeta_k(t)   \cE^{(t)}(u(t) , \xi_R\Psi(t)) \d t,\\
    \int_0^{T} \zeta_k(t) \int_{\R^d} \xi_R(x) \partial_t \Psi(t,x) u(t,x)  \d x \d t&=\int_0^{T} \zeta_k(t)   \cE^{(t)}(\Psi(t) , \xi_R u(t) )  \d t.
\end{align*}
To continue we plug these two identities back into the previous equation, recall the definition of $\cE^{(t)}$, and employ the following algebraic identity
    \begin{align}
    \label{eq:algebra-chain-rule}\nonumber
        (\phi_1 - \phi_2)(u_1 \xi_1 - u_2 \xi_2) &= \frac{1}{2}(\phi_1 - \phi_2) (u_1 - u_2)(\xi_1 + \xi_2) + \frac{1}{2}(\phi_1 - \phi_2) (u_1 + u_2)(\xi_1 - \xi_2) \\ \nonumber
        &= (u_1 - u_2)(\phi_1 \xi_1 - \phi_2 \xi_2) - \frac{1}{2}(u_1 - u_2)(\xi_1 - \xi_2)(\phi_1 + \phi_2) \\ 
        &\quad+ \frac{1}{2}(\phi_1 - \phi_2)(u_1 + u_2)(\xi_1 - \xi_2).
    \end{align}
As a result, we arrive at
\begin{align*}
     \bigg| \int_0^{T} &\partial_t \zeta_k(t) \int_{\R^d} \xi_R(x) u(t,x) \Psi(t,x) \d x \d t \bigg| \\
    & \le \frac{1}{2}\bigg| \int_0^{T} \zeta_k(t) \bigg( \int_{\R^d} \int_{\R^d} (u(t,x) - u(t,y))(\Psi(t,x) + \Psi(t,y)) \\
    &\qquad\qquad\qquad\qquad\qquad\qquad\cdot (\xi_R(x) - \xi_R(y)) K(t;x,y) \d y \d x \bigg) \d t \bigg| \\
    &\quad \quad + \frac{1}{2} \bigg| \int_0^{T} \zeta_k(t) \bigg( \int_{\R^d} \int_{\R^d} (u(t,x) + u(t,y))(\Psi(t,x) - \Psi(t,y))\\
    &\qquad\qquad\qquad\qquad\qquad\qquad\cdot(\xi_R(x) - \xi_R(y)) K(t;x,y) \d y \d x \bigg)\d t \bigg| \\
    &=: \textbf{I} + \textbf{II}.
\end{align*}

Let us treat the above two terms separately. 
For $\mathbf{I}$, we estimate using the symmetry of the integrand, the properties of $\xi_R,\,\z_k$ and the upper bound for $\Psi$ from \autoref{lemma:Psi-bounds}:

\begin{align*}
    \mathbf{I} \le c \int_0^T \iint_{(B_R^c\times B_R^c)^c} \frac{|u(t,x) - u(t,y)|}{(1 + |x|)^{d+2s}}   (1 \wedge R^{-1}|x-y|) K(t;x,y)  \d y \d x \d t =: c\,\mathbf{I}_R.
\end{align*}

For $\mathbf{II}$, we estimate by using the symmetry of the integrand and the properties of $\xi_R,\eta_k$:
\begin{align*}
    \mathbf{II} &\le c\int_{t_o}^{T}  \iint_{(B_R^c\times B_R^c)^c} u(t,x) |\Psi(\tau,x) - \Psi(t,y)| (1 \wedge R^{-1}|x-y|)K(t;x,y) \d y \d x   \d t\\
    &=: c\,\mathbf{II}_R.
\end{align*}

Hence, combining the estimates for $\mathbf{I}$, $\mathbf{II}$
and taking the limit $k \to \infty$, we deduce
\begin{align*}
    \bigg| \int_{\R^d} u(\tau_o , x) \xi_R(x) \Psi(\tau_o,x) \d x - \int_{\R^d} u(\tau_1 , x) \xi_R(x) \Psi(\tau_1,x) \d x  \bigg| \le c \, \mathbf{I}_R + c\, \mathbf{II}_R
\end{align*}
for any $R>1$ and for any $\tau_o,\,\tau_1$ satisfying $t_o \le \tau_o < \tau_1 < T$.
We will show in \autoref{lemma:I1} and \autoref{lemma:I2} that $\mathbf{I}_R,\, \mathbf{II}_R \to 0$, as $R \to \infty$.
Hence, letting $R \to \infty$ in the last estimate, we obtain
\begin{align*}
    \int_{\R^d} u(\tau_o , x) \Psi(\tau_o,x) \d x = \int_{\R^d} u(\tau_1 , x)  \Psi(\tau_1,x) \d x.
\end{align*}

The proof is complete upon applying the pointwise upper and lower estimate on $\Psi$ established in \autoref{lemma:Psi-bounds}.
\end{proof}

It remains to state and prove \autoref{lemma:I1} and \autoref{lemma:I2}, which were employed during the proof of \autoref{lemma:weighted-L1-time-insensitive}.

\begin{lemma}
    \label{lemma:I1}
    Assume that the kernel $K$ satisfies the upper bound in \eqref{eq:Kcomp} and \eqref{eq:coercive}.
Let $u$ be a nonnegative, global weak solution to \eqref{Eq:1:1} in $(0,T) \times \R^d$. Define
\begin{align*}
    \mathbf{I}_R := \int_0^T \iint_{(B_R^c\times B_R^c)^c} \frac{|u(t,x) - u(t,y)|}{(1 + |x|)^{d+2s}}   (1 \wedge R^{-1}|x-y|) K(t;x,y) \d y \d x \d t.
\end{align*}
    Then, we have $\displaystyle \lim_{R \to \infty} \mathbf{I}_R = 0$.
\end{lemma}

\begin{proof}
For ease of notation we will consider $\mathbf{I}_{R/2}$ and split the domain of integration as $ (B_{R}\times B_{R}) \cup (B_{R/2}\times B^c_{R})\cup (B^c_{R} \times B_{R/2}) \supset (B_{R/2}^c\times B_{R/2}^c)^c$.
As a result, we estimate
\begin{align*}
    \mathbf{I}_{R/2} &\le \int_0^T \int_{B_R} \frac{1}{(1 + |x|)^{d+2s}} \int_{B_R} |u(t,x) - u(t,y)| (1 \wedge R^{-1}|x-y|) K(t;x,y)  \d y \d x \d t \\
    & \quad + \int_0^T \int_{B_{R/2}} \frac{u(t,x)}{(1 + |x|)^{d+2s}} \int_{B_R^c}  \frac{1 \wedge R^{-1}|x-y|}{|x-y|^{d+2s}}  \d y \d x \d t \\
    & \quad + \int_0^T \int_{B_{R/2}} \frac{1}{(1 + |x|)^{d+2s}} \int_{B_R^c} u(t,y) \frac{1 \wedge R^{-1}|x-y|}{|x-y|^{d+2s}}  \d y \d x \d t \\
    &\quad + \int_0^T \int_{B_R^c} \frac{u(t,x)}{(1 + |x|)^{d+2s}} \int_{B_{R/2}}  \frac{1 \wedge R^{-1}|x-y|}{|x-y|^{d+2s}}  \d y \d x \d t \\
     &\quad + \int_0^T \int_{B_R^c} \frac{1}{(1 + |x|)^{d+2s}} \int_{B_{R/2}} u(t,y) \frac{1 \wedge R^{-1}|x-y|}{|x-y|^{d+2s}}  \d y \d x \d t \\
    &=: \textbf{I}_1 + \mathbf{I}_{2} + \mathbf{I}_{3} + \mathbf{I}_{4} + \mathbf{I}_{5}.
\end{align*}
While the treatment of $\mathbf{I}_{2} $, $\mathbf{I}_{3} $, $ \mathbf{I}_{4} $ and $ \mathbf{I}_{5}$ is straightforward, estimating $\mathbf{I_1}$ requires significant effort. Let us treat $\mathbf{I}_{1}$ at the end of the proof. For $\mathbf{I}_{2}$, we obtain from the upper bound in \eqref{eq:Kcomp}
\begin{align*}
        \mathbf{I}_{2} &\le \int_0^T \int_{B_{R/2}} \frac{u(t,x)}{(1+|x|)^{d+2s}} \bigg(\int_{B_{R}^c} \frac{\d y}{|x-y|^{d+2s}}  \bigg) \d x \d t \le \frac{c}{R^{2s}} \|u\|_{L^1((0,T);L^1(\R^d;\mu))}
    \end{align*}
for some $c=c(d,s)$.
 For $\mathbf{I}_{3}$, we use that when $x \in B_{R/2}$ and $y \in B_R^c$, then $|y| \le |x-y| + |x| \le |x-y| + \frac{R}{2} \le 2 |x-y|$ and thus
    \begin{align*}
        \mathbf{I}_{3} &\le c\, \mu(\R^d) \int_0^T \int_{B_{R}^c}\frac{u(t,y)}{|y|^{d+2s}} \d y \d t 
    \end{align*}
for some $c=c(d)$.
 For $\mathbf{I}_{4}$, we observe that $B_{R/2}\subset B^c_{R/2}(x)$ when $x\in B_R^c$ and estimate
    \begin{align*}
        \mathbf{I}_{4} &\le \int_0^T \int_{B_{R}^c} \frac{u(t,x)}{(1+|x|)^{d+2s}}  \bigg( \int_{B^c_{R/2}(x)} \frac{\d y}{|x-y|^{d+2s}} \bigg) \d x \d t \le \frac{c}{R^{2s}} \|u\|_{L^1((0,T);L^1(\R^d;\mu))}
    \end{align*}
for some $c=c(d,s)$.
 Finally, for $\mathbf{I}_{5}$, we observe
    \begin{align*}
        \mathbf{I}_{5} &\le c \bigg(\int_{B_{R}^c} \frac{\d x}{|x|^{d+2s}} \bigg) \int_0^T \int_{B_{R/2}} \frac{u(t,y)}{ R^{d+2s}} \d y \d t \le \frac{c}{R^{2s}} \|u\|_{L^1((0,T);L^1(\R^d;\mu))}
    \end{align*}
for some $c=c(d,s)$.
    Therefore, $\mathbf{I}_{j}$, $j=2,3,4,5$ all tend to $0$ as $R\to\infty$, thanks to the assumption that $u \in L^1((0,T);L^1(\R^d;\mu))$.

    Let us now turn to the estimate for $\textbf{I}_1$.

    First, observe that by \autoref{lemma:energy-decay} and \autoref{Lm:Lq-Lsig} with $\sig=1$, we have for any $x_o \in \R^d$:
\begin{align*}
    \frac{1}{R} \int_0^T &\int_{B_5(x_o)} \int_{B_5(x_o)} |u(t,x) - u(t,y)| K(t;x,y) |x-y| \d y \d x \d t \le \frac{c}{R} \int_0^{2T} \int_{B_{10}(x_o)} u(t,x) \d x \d t
\end{align*}
for some $c=c(d,s,\lm,\Lm,T) > 0$.
In particular, this implies
\begin{align}
\label{eq:on-diag-1}
\begin{split}
   \frac{1}{R} \int_0^T \int_{B_5(x_o)} \frac1{(1 + |x|)^{d+2s}}  \int_{B_5(x_o)}& |u(t,x) - u(t,y)| K(t;x,y) |x-y| \d y \d x \d t \\
    &\le \frac{c}{R} \int_0^{2T} \int_{B_{10}(x_o)} \frac{u(t,x)}{ (1+ |x|)^{d+2s} }\d x \d t,
\end{split}
\end{align}
for some $c=c(d,s,\lm,\Lm,T) > 0$,
since for any $x\in B_{10}(x_o)$ we have $\frac1{11}(1+|x|) \le (1+|x_o|)\le 11(1+|x|) $. Here, the extension of time interval to $(0,2T)$ is for ease of notation.

In order to conclude the proof, let us perform a covering argument. 
Indeed, consider the countable set
\begin{equation*}
    \mathcal{S}:=\Big\{\frac12(n_1,\dots , n_d):\,  n_i\in\Z,\,i=1,\dots, d\Big\}.
\end{equation*}
For $R > 100$ define $\mathcal{S}_R:=\mathcal{S}\cap B_R$ to be a finite subset with cardinality $\#\mathcal{S}_R\le (2R)^d$. It is not hard to see that for some $C=C(d)>0$, it holds
\begin{align}\label{Eq:cover-balls}
    B_{R} \subset \bigcup_{x_o \in \mathcal{S}_R} B_1(x_o), \qquad \sum_{x_o \in \mathcal{S}_R} \chi_{B_{10}(x_o)} \le C \chi_{B_{20R}}.
\end{align}
Using the first property of \eqref{Eq:cover-balls} we obtain
\begin{align*}
    \textbf{I}_1 &\le \sum_{x_o,\, y_o\in\mathcal{S}_R } \int_0^T\int_{B_{1}(x_o)} \frac1{(1 + |x|)^{d+2s}}\int_{B_{1}(y_o)} |u(t,x) - u(t,y)| (1 \wedge R^{-1}|x-y|) K(t;x,y) \d y \d x \d t .
\end{align*}
To proceed, let us split the summation over $x_o,\,y_o\in\mathcal{S}_R$ by considering \textbf{two cases}, that is, $|x_o-y_o|<4$ and $|x_o-y_o|\ge4$. Let us deal with the \textbf{first case}. Indeed, when $|x_o-y_o|<4$, we observe that $B_1(y_o)\subset B_5(x_o)$ and  estimate by enlarging the integration domains:
\begin{align}\label{Eq:case:1}\nonumber
&\sum_{\substack{x_o,\, y_o\in\mathcal{S}_R\\ |x_o-y_o|<4}} \int_0^T\int_{B_{1}(x_o)} \frac1{(1 + |x|)^{d+2s}}\int_{B_{1}(y_o)} |u(t,x) - u(t,y)| (1 \wedge R^{-1}|x-y|) K(t;x,y) \d y \d x \d t \\ \nonumber
&\quad\le \frac{8^d}R\sum_{x_o\in\mathcal{S}_R}\int_0^T\int_{B_{5}(x_o)} \frac1{(1 + |x|)^{d+2s}}\int_{B_{5}(x_o)}  |u(t,x) - u(t,y)| K(t;x,y) |x-y| \d y \d x \d t\\ \nonumber
&\quad\le \frac{c}{R}\sum_{x_o\in\mathcal{S}_R}\int_0^{2T}\int_{B_{10}(x_o)} \frac{u(t,x)}{(1 + |x|)^{d+2s}} \d x \d t\\ 
&\quad\le \frac{c}{R} \|u\|_{L^1((0,2T);L^1(\R^d;\mu))}
\end{align}
for some $c=c(d,s,\lm, \Lm, T)>0$.
Here, in the second-to-last estimate we used \eqref{eq:on-diag-1} and in the last estimate we employed the second property of \eqref{Eq:cover-balls}.

To handle the \textbf{second case}, 
 let us consider $x_o,\, y_o \in \R^d$ satisfying $|x_o - y_o| \ge 4$. For any $x\in B_1(x_o)$ and $y\in B_1(y_o)$, we observe that $\frac23|x-y|\le |x_o-y_o|\le 2|x-y|$ and $\frac12(1+|y|)\le 1+|y_o|\le 2(1+|y|)$. Hence, we  estimate
\begin{align}\label{J1-J2}\nonumber
     \int_0^T \int_{B_1(x_o)}& \frac1{(1 + |x|)^{d+2s}} \int_{B_1(y_o)} |u(t,x) - u(t,y)| (1 \wedge R^{-1}|x-y|) K(t;x,y) \d y \d x \d t \\ \nonumber
    &  \le \int_0^T \int_{B_1(x_o)} \frac{u(t,x)}{(1 + |x|)^{d+2s}}  \bigg( \int_{B_1(y_o)} \frac{1 \wedge R^{-1} |x-y|}{|x-y|^{d+2s}} \d y \bigg) \d x \d t \\ \nonumber
    &  \quad + \int_0^T \int_{B_1(x_o)} \frac1{(1 + |x|)^{d+2s}} \int_{B_1(y_o)} u(t,y)  \frac{1 \wedge R^{-1} |x-y|}{|x-y|^{d+2s}} \d y \d x \d t \\ \nonumber
    &  \le c  \frac{1 \wedge R^{-1}|x_o - y_o| }{|x_o - y_o|^{d+2s}} \int_0^T \int_{B_1(x_o)} \frac{ u(t,x)}{(1 + |x|)^{d+2s}} \d x \d t \\ \nonumber
    &  \quad + c \bigg(\frac{1+|y_o|}{1 + |x_o|} \bigg)^{d+2s}  \frac{1 \wedge R^{-1}|x_o - y_o| }{|x_o - y_o|^{d+2s}}\int_0^T\int_{B_1(y_o)} \frac{u(t,y)}{(1 + |y|)^{d+2s}} \d y \d t \\
    &  =: J_1(x_o,y_o) + J_2(x_o,y_o)
\end{align}
for some $c=c(d)$.
Therefore, taking \eqref{J1-J2} into account, the second case leads to
\begin{align}\label{Eq:case:2} \nonumber
\sum_{\substack{x_o,\, y_o\in\mathcal{S}_R\\ |x_o-y_o|\ge4}} &\int_0^T\int_{B_{1}(x_o)} \frac1{(1 + |x|)^{d+2s}}\int_{B_{1}(y_o)} |u(t,x) - u(t,y)| (1 \wedge R^{-1}|x-y|) K(t;x,y) \d y \d x \d t \\
   & \le\sum_{\substack{x_o,\, y_o\in\mathcal{S}_R\\ |x_o-y_o|\ge4}} J_1(x_o,y_o) + \sum_{\substack{x_o,\, y_o\in\mathcal{S}_R\\ |x_o-y_o|\ge4}} J_2(x_o,y_o) =: J_1 + J_2.
\end{align}
Let us consider $J_1$ and $J_2$, separately. First for $J_1$, we observe that
when $|x_o-y_o|\ge 4$, it holds $\frac34|x_o-y_o|\le |x_o-y|\le\frac54|x_o-y_o|$ for any $y\in B_1(y_o)$, and moreover, we have $B_1(y_o)\subset\R^d\setminus B_1(x_o)$. Hence, we estimate
 for fixed $x_o\in\mathcal{S}_R$ that
\begin{align}\nonumber
    \sum_{\substack{y_o \in \mathcal{S}_{R}\\ |x_o - y_o| \ge 4}} & \frac{ 1 \wedge R^{-1}|x_o - y_o|}{|x_o - y_o|^{d+2s}} \le c \sum_{\substack{y_o \in \mathcal{S}_{R}\\ |x_o - y_o| \ge 4}} \int_{B_1(y_o)} \frac{ 1 \wedge R^{-1}|x_o - y|}{|x_o - y|^{d+2s}} \d y \\ \nonumber
    &\le c\int_{\R^d \setminus B_1(x_o)} \bigg(\sum_{y_o \in \mathcal{S}_{R}} \chi_{B_1(y_o)} \bigg) \frac{ 1 \wedge R^{-1}|x_o - y|}{|x_o - y|^{d+2s}} \d y\\ \nonumber
    &\overset{\eqref{Eq:cover-balls}}{\le} c\int_{\R^d \setminus B_1(x_o)} \frac{ 1 \wedge R^{-1}|x_o - y|}{|x_o - y|^{d+2s}} \d y \\ \nonumber
    &\le \frac{c}{R} \int_{B_R(x_o) \setminus B_1(x_o)} \frac{\d y}{|x_o - y|^{d+2s-1}}  + c \int_{\R^d \setminus B_R(x_o)} \frac{\d y}{|x_o - y|^{d+2s}}  \\ \label{Eq:aug-est-1}
    &\le c \frac{\mathfrak{C}(R,s)}R + \frac{c}{R^{2s}}
\end{align}
for some $c=c(d)$, where we introduced
\begin{equation}\label{Eq:C(R,s)}
    \mathfrak{C}(R,s):=\left\{
    \begin{array}{cl}
        \displaystyle\frac{R^{1-2s}}{1-2s} \quad&\text{if}\>\> s\in(0,\tfrac12),  \\[10pt]
         \ln R \quad&\text{if}\>\> s=\tfrac12, \\[5pt]
         \displaystyle\frac1{2s-1} \quad& \text{if}\>\>s\in(\tfrac12,1).
    \end{array}\right.
\end{equation}
Therefore, we estimate $J_1$ as
\begin{align}\label{Eq:J1}\nonumber
J_1 &=
    \sum_{x_o \in \mathcal{S}_{R}} \int_0^T \int_{B_1(x_o)}  \frac{u(t,x)}{(1 + |x|)^{d+2s}} \d x \d t \Bigg( \sum_{\substack{y_o \in \mathcal{S}_{R}\\ |x_o - y_o| \ge 4}} \frac{1 \wedge R^{-1}|x_o - y_o|}{|x_o - y_o|^{d+2s}} \Bigg) \\ \nonumber
    &\le c \bigg(\frac{\mathfrak{C}(R,s)}R + \frac{1}{R^{2s}} \bigg) \sum_{x_o \in \mathcal{S}_{R}} \int_0^T \int_{B_1(x_o)} \frac{u(t,x)}{(1 + |x|)^{d+2s}} \d x \d t \\
    &\le c \bigg(\frac{\mathfrak{C}(R,s)}R + \frac{1}{R^{2s}}\bigg) \Vert u \Vert_{L^1((0,T); L^1(\R^d;\mu))}  
\end{align}
for some $c=c(d)$.
 Here, we employed the second property of \eqref{Eq:cover-balls} to get the last estimate.
 
For $J_2$, we estimate $(1 + |y_o|) \le (1 + |x_o|) + |x_o - y_o|$ and obtain
\begin{align*}
 J_2(x_o,y_o) &\le c\frac{1 \wedge R^{-1}|x_o - y_o|}{|x_o - y_o|^{d+2s}}  \int_0^T\int_{B_1(y_o)} \frac{u(t,y)}{(1 + |y|)^{d+2s}} \d y \d t \\
    &\quad + c\frac{1 \wedge R^{-1} |x_o - y_o|}{(1 + |x_o|)^{d+2s}}  \int_0^T\int_{B_1(y_o)} \frac{u(t,y)}{(1 + |y|)^{d+2s}} \d y \d t \\
    &=: J_{2,1}(x_o,y_o) + J_{2,2}(x_o,y_o).
\end{align*}
Note that the estimate for the sum over $J_{2,1}$ proceeds in the exact same way as the estimate for the sum over $J_1$, up to exchanging the roles of $x_o$ and $y_o$. After all, we obtain
\begin{align}\label{Eq:J21}
    \sum_{\substack{x_o,\, y_o\in\mathcal{S}_R\\ |x_o-y_o|\ge4}} J_{2,1}(x_o,y_o) \le c(d)\bigg(\frac{\mathfrak{C}(R,s)}R + \frac{1}{R^{2s}}\bigg) \Vert u \Vert_{L^1((0,T); L^1(\R^d;\mu))}  .
\end{align}
Hence, it remains to estimate
\begin{align}\label{Eq:J22}\nonumber
    \sum_{\substack{x_o,\, y_o\in\mathcal{S}_R\\ |x_o-y_o|\ge4}} & J_{2,2}(x_o,y_o) \\
    &=  \sum_{y_o \in \mathcal{S}_{R}} \int_0^T\int_{B_1(y_o)} \frac{u(t,y)}{ (1 + |y|)^{d+2s}} \d y \d t \Bigg(\sum_{\substack{x_o \in \mathcal{S}_R \\ |x_o - y_o| \ge 4}} \frac{1 \wedge R^{-1} |x_o - y_o| }{(1 + |x_o|)^{d+2s}} \Bigg) .
\end{align}
In order for that, let $K > 100$ to be chosen later, and take $R > 2K$. We consider two complimentary possibilities, i.e. $y_o\in\mathcal{S}_K$ and $y_o\in\mathcal{S}_R\setminus\mathcal{S}_K$. In the first possibility, let us fix $y_o\in\mathcal{S}_K$ and observe that, when $|x_o-y_o|\ge4$ we have $\frac34|x_o-y_o|\le |x-y_o|\le\frac54|x_o-y_o|$ for any $x\in B_1(x_o)$. As a result, we estimate
\begin{align*}
     \sum_{\substack{x_o\in\mathcal{S}_R\\ |x_o-y_o|\ge4}} \frac{1 \wedge R^{-1} |x_o - y_o| }{(1 + |x_o|)^{d+2s}} & \le c(d)\sum_{x_o \in \mathcal{S}_{R}}  \int_{B_1(x_o)}\frac{1 \wedge R^{-1}|x - y_o|}{ (1 + |x|)^{d+2s}} \d x\\
    &\overset{\eqref{Eq:cover-balls}}{\le} c(d) \int_{\R^d} \frac{1 \wedge R^{-1}|x - y_o|}{ (1 + |x|)^{d+2s}} \d x .
\end{align*}
To continue estimating the integral in the above display, we split the domain of integration into $\R^d = B_{R} \cup B_R^c$, use the triangle inequality $|x - y_o| \le |x| + |y_o|$, and obtain
\begin{align*}
\sum_{\substack{x_o\in\mathcal{S}_R\\ |x_o-y_o|\ge4}} &\frac{1 \wedge R^{-1} |x_o - y_o| }{(1 + |x_o|)^{d+2s}}\le c \int_{B_R} \frac{ 1 \wedge R^{-1}|x - y_o|}{ (1 + |x|)^{d+2s}} \d x+c \int_{B_R^c} \frac{ 1 \wedge R^{-1}|x - y_o|}{ (1 + |x|)^{d+2s}} \d x\\
&\le c \int_{B_R} \frac{ R^{-1}|x - y_o|}{ (1 + |x|)^{d+2s}} \d x+c \int_{B_R^c} \frac{ \d x }{ (1 + |x|)^{d+2s}} \\
        & \le \frac{c}{R} \int_{B_{R}} \frac{\d x}{(1 + |x|)^{d+2s-1}}  + c\frac{ |y_o|}{R} \int_{B_{R}} \frac{\d x}{(1 + |x|)^{d+2s} } + c\int_{ B^c_{R}} \frac{\d x}{(1 + |x|)^{d+2s} }.
\end{align*}
Here, the constant $c=c(d)$ is the same as in the previous estimate.
Among these three integrals on the right-hand side, it is clear that the last integral can be estimated by $c(d,s)/R^{2s}$, while the second integral is bounded by $\mu(\R^d)$. For the first integral, we estimate instead
\begin{align*}
    \int_{B_{R}} \frac{\d x}{(1 + |x|)^{d+2s-1}}&=\int_{B_{R}\setminus B_1} \frac{\d x}{(1 + |x|)^{d+2s-1}}+\int_{B_{1}} \frac{\d x}{(1 + |x|)^{d+2s-1}}\\
    &\le \int_{B_{R}\setminus B_1} \frac{\d x}{|x|^{d+2s-1}} + |B_1|\\
    &\le \mathfrak{C}(R,s)+|B_1|,
\end{align*}
where $\mathfrak{C}(R,s)$ is as in \eqref{Eq:C(R,s)}.
Having these considerations, we can continue and estimate
\begin{align*}
    \sum_{\substack{x_o\in\mathcal{S}_R\\ |x_o-y_o|\ge4}} \frac{1 \wedge R^{-1} |x_o - y_o| }{(1 + |x_o|)^{d+2s}}\le c(d)\frac{\mathfrak{C}(R,s)+|B_1|}{R} + c(d) \frac{ |y_o|}{R}+ \frac{c(d,s)}{R^{2s}}.
\end{align*}

Thus, for any $K>100$ and $R>2K$ we have that 
\begin{align*}
     \sum_{\substack{x_o\in\mathcal{S}_R,\, y_o \in  \mathcal{S}_K\\ |x_o-y_o|\ge4}}  J_{2,2}(x_o,y_o) &\le c \bigg[\frac{\mathfrak{C}(R,s)}{R}+\frac{1}{R^{ 2s}}  +  \frac{K}{R}\bigg] \sum_{y_o\in\mathcal{S}_K}\int_0^T\int_{B_1(y_o)} \frac{u(t,y)}{ (1 + |y|)^{d+2s}} \d y \d t \\
     &\overset{\eqref{Eq:cover-balls}}{\le} c \bigg[\frac{\mathfrak{C}(R,s)}{R}+\frac{1}{R^{ 2s}} +  \frac{K}{R}\bigg]\Vert u \Vert_{L^1((0,T);L^1(\R^d;\mu))}
\end{align*}
for some  $c=c(d,s)$.
Next, for the second possibility, i.e. when $y_o \in \mathcal{S}_{R} \setminus \mathcal{S}_K$, we estimate by the exact same argument as before and conclude
\begin{align*}
    \sum_{\substack{x_o\in\mathcal{S}_R\\ |x_o-y_o|\ge4}} \frac{1 \wedge R^{-1} |x_o - y_o|}{(1 + |x_o|)^{d+2s} } \le c(d)\frac{\mathfrak{C}(R,s)+|B_1|}R + c(d) \frac{ |y_o|}{R}+ \frac{c(d,s)}{R^{2s}} \le c(d,s).
\end{align*}
Observing also that $B_1(y_o)\subset   B^c_{K-1}$ whenever $y_o \in \mathcal{S}_{R}\setminus \mathcal{S}_K$, we then obtain
\begin{align*}
    \sum_{\substack{x_o\in\mathcal{S}_R,\, y_o \in \mathcal{S}_{R} \setminus \mathcal{S}_K\\ |x_o-y_o|\ge4}} J_{2,2}(x_o,y_o) &\le c \sum_{y_o \in \mathcal{S}_{R}\setminus \mathcal{S}_K} \int_0^T\int_{B_1(y_o)} \frac{u(t,y)}{ (1 + |y|)^{d+2s}} \d y \d t\\
    &\le c  \int_0^T\int_{ B^c_{K-1}} \bigg(\sum_{y_o \in \mathcal{S}_{R} }\chi_{B_1(y_o)} \bigg) \frac{u(t,y)}{ (1 + |y|)^{d+2s}} \d y \d t\\
    &\overset{\eqref{Eq:cover-balls}}{\le} c \bigg[ \int_0^T \int_{  B^c_{K-1}} \frac{u(t,y)}{ (1 + |y|)^{d+2s}} \d y \d t \bigg]
\end{align*}
for some  $c=c(d)$.
We collect these estimates for both possibilities in \eqref{Eq:J22} and arrive at 
\begin{align}\label{Eq:J22+}\nonumber
    \sum_{\substack{x_o,\, y_o\in\mathcal{S}_R\\ |x_o-y_o|\ge4}} J_{2,2}(x_o,y_o) 
    &\le c \bigg[\frac{\mathfrak{C}(R,s)}{R}+\frac{1}{R^{ 2s}} +  \frac{K}{R}\bigg]\Vert u \Vert_{L^1((0,T);L^1(\R^d;\mu))}\\
    &\quad+c \bigg[ \int_0^T \int_{ B^c_{K-1}} \frac{u(t,y)}{ (1 + |y|)^{d+2s}} \d y \d t \bigg]
\end{align}
for some  $c=c(d,s)$.
Substituting the estimates \eqref{Eq:J1}, \eqref{Eq:J21} and \eqref{Eq:J22+} in \eqref{Eq:case:2} finishes the analysis of the \textbf{second case}. Combining this with the estimate \eqref{Eq:case:1} from the \textbf{first case}, we conclude that for any $K>100$ and any $R>2K$, it holds 
\begin{align*}
    \mathbf{I}_1&\le c\bigg[\frac{\mathfrak{C}(R,s)}{R}+\frac{1}{R^{ 2s}} +  \frac{K}{R}\bigg] \Vert  u \Vert_{L^1((0,2T) ; L^1(\R^d;\mu))}+c \bigg[ \int_0^{T} \int_{ B^c_{K-1}} \frac{u(t,y)}{ (1 + |y|)^{d+2s}} \d y \d t \bigg]
\end{align*}
for some $c=c(d,s,\lm, \Lm, T)$.
Note that the second quantity on the right-hand side can be made smaller than any $\eps > 0$ by choosing $K$ large enough, thanks to the assumption $u \in L^1((0,2T);L^1(\R^d;\mu))$. 
After $K$ is fixed, we choose $R$ large enough, such that the first quantity is also estimated by $\varep$, and hence $\mathbf{I}_1 \le 2\eps$. The proof is complete.
\end{proof}

\begin{lemma}
    \label{lemma:I2}
    Assume that the kernel $K$ satisfies the upper bound in \eqref{eq:Kcomp} and \eqref{eq:coercive}.
Let $u$ be a nonnegative, global weak solution to \eqref{Eq:1:1} in $(0,T) \times \R^d$, and $\Psi$ be a weak solution to
\begin{align*}
\begin{cases}
    \partial_t \Psi + \mathcal{L}_t \Psi &= 0 ~~ \text{ in } (0,T) \times \R^d,\\
    \Psi(T,\cdot) &= (1 + |\cdot|)^{-d-2s} ~~ \text{ in } \R^d.
    \end{cases}
\end{align*}
Fix $t_o\in(0, T)$ and define
\begin{align*}
    \mathbf{II}_R := \int_{t_o}^{T}  \iint_{(B_R^c\times B_R^c)^c} u(t,x) |\Psi(t,x) - \Psi(t,y)| (1 \wedge R^{-1}|x-y|) K(t;x,y) \d y \d x   \d t.
\end{align*}
    Then, we have $\displaystyle\lim_{R \to \infty} \mathbf{II}_R = 0$.
\end{lemma}

\begin{proof}
    By splitting the domain of integration into several parts as in \autoref{lemma:I1}, and recalling the pointwise upper bound for $\Psi$ from \autoref{lemma:Psi-bounds}
    \begin{align}
    \label{eq:Psi-upper}
        \Psi(t,x) \le \frac{c}{(1 + |x|)^{d+2s}}    
    \end{align}
    for some $c=c(d,s,\lm, \Lm, T)$,
    we estimate
\begin{align*}
    \mathbf{II}_{R/2} &\le \int_{t_o}^{T} \bigg( \int_{B_{R}} u(t,x) \int_{B_{R}} |\Psi(t,x) - \Psi(t,y)| (1 \wedge R^{-1}|x-y|) K(t;x,y) \d y \d x \bigg) \d t \\
    & \quad + c\int_{t_o}^{T} \int_{B_{R/2}} \frac{u(t,x)}{(1 + |x|)^{d+2s}} \int_{B_R^c}  \frac{1 \wedge R^{-1}|x-y|}{|x-y|^{d+2s}}  \d y \d x \d t \\
    & \quad + c\int_{t_o}^{T} \int_{B_{R/2}} u(t,x) \int_{B_R^c} \frac{1}{(1 + |y|)^{d+2s}}  \frac{1 \wedge R^{-1}|x-y|}{|x-y|^{d+2s}}  \d y \d x \d t \\
    &\quad + c\int_{t_o}^{T} \int_{B_R^c} \frac{u(t,x)}{(1 + |x|)^{d+2s}} \int_{B_{R/2}}  \frac{1 \wedge R^{-1}|x-y|}{|x-y|^{d+2s}}  \d y \d x \d t \\
     &\quad + c\int_{t_o}^{T} \int_{B_R^c} u(t,x) \int_{B_{R/2}}  \frac{1}{(1 + |y|)^{d+2s}}  \frac{1 \wedge R^{-1}|x-y|}{|x-y|^{d+2s}}  \d y \d x \d t \\
    &=: \textbf{II}_1 + \mathbf{II}_2 + \mathbf{II}_3 + \mathbf{II}_4 + \mathbf{II}_5.
\end{align*}
While the treatment of $\mathbf{II}_2 $, $ \mathbf{II}_3 $, $ \mathbf{II}_4 $ and $ \mathbf{II}_5$ is straightforward, estimating $\mathbf{II}_1$ requires significant effort. Let us treat $\mathbf{II}_1$ at the end of the proof. Moreover, we observe that $\mathbf{II}_2$ and $\mathbf{II}_4$ coincide with the quantities $\mathbf{I}_2$ and $\mathbf{I}_4$ from the proof of \autoref{lemma:I1}. Consequently, they tend to $0$ as $R\to\infty$.
As for $\mathbf{II}_3$, we observe that when $x \in B_{R/2}$ and $y \in B_R^c$, there holds $|y| \le |x-y| + |x| \le |x-y| + \frac{R}{2} \le 2 |x-y|$, and thus we may estimate
    \begin{align*}
        \mathbf{II}_3 &\le c\int_{t_o}^{T} \int_{B_{R/2}} u(t,x) \int_{B_R^c} \frac1{(1 + |y|)^{d+2s} |y|^{d+2s}}  \d y \d x \d t \\
        &\le c\int_{t_o}^{T} \int_{B_{R/2}} \frac{u(t,x)}{ R^{d+4s}} \d x \d t \le \frac{c}{R^{2s}} \Vert u \Vert_{L^1(0,T;L^1(\R^d;\mu))}
    \end{align*}
for some $c=c(d,s,\lm, \Lm, T)$.
 Finally, for $\mathbf{II}_5$, we observe that
    \begin{align*}
        \mathbf{II}_5 &\le c\int_{t_o}^{T} \int_{B_{R/2}} \frac{1}{(1 + |y|)^{d+2s}} \int_{B_R^c}  \frac{u(t,x)}{|x-y|^{d+2s}}  \d y \d x \d t \le c\int_{t_o}^{T} \int_{B_R^c} \frac{u(t,x)}{ |x|^{d+2s} }\d x \d t 
    \end{align*}
for some $c=c(d,s,\lm, \Lm, T)$.
    Therefore, $\mathbf{II}_3$ and $\mathbf{II}_5$ tend to $0$ as $R \to \infty$, since by assumption $u \in L^1((0,T);L^1(\R^d;\mu))$.

    Let us now turn to the estimate for $\textbf{II}_1$. 
The proof is reminiscent of the treatment of $\textbf{I}_1$ in the proof of \autoref{lemma:I1} in the sense that we will again use the covering argument based on \eqref{Eq:cover-balls}.
Indeed, we first use the covering properties in \eqref{Eq:cover-balls} to estimate
    \begin{align*}
    \textbf{II}_1 &\le \sum_{x_o,\, y_o\in\mathcal{S}_R } \int_{t_o}^{T}   \int_{B_{1}(x_o)} u(t,x) \int_{B_{1}(y_o)} |\Psi(t,x) - \Psi(t,y)| (1 \wedge R^{-1}|x-y|) K(t;x,y) \d y \d x   \d t.
\end{align*}
The summation over $x_o,\,y_o\in \mathcal{S}_R$ is divided into \textbf{two cases}, that is, $|x_o-y_o|<4$ and $|x_o-y_o|\ge4$. To deal with the \textbf{first case}, we use \autoref{Lm:u-Psi} to get
    \begin{align*}
    \int_{t_o}^{T}  &\int_{B_5(x_o)} \int_{B_5(x_o)}  |\Psi(t,x) - \Psi(t,y)| K(t;x,y) |x-y|  \d y \d x \d t\le \frac{c}{(1+ |x_o|)^{d+2s}} 
\end{align*}
for some $c=c(d,s,\lm,\Lm, T)$.
This joint with the boundedness estimate in \autoref{Prop:bd} yields an analog of \eqref{eq:on-diag-1}: 
    \begin{align}
    \label{eq:on-diag-2}
    \begin{split}
    \frac{1}{R} \int_{t_o}^{T}  \int_{B_5(x_o)} \int_{B_5(x_o)}& u(t,x) |\Psi(t,x) - \Psi(t,y)| K(t;x,y) |x-y|  \d y \d x \d t \\
    &\le \frac{c}{R} \int_{0}^{T} \int_{B_{10}(x_o)} \frac{u(t,x)}{(1 + |x|)^{d+2s}} \d x \d t\\
    &\quad+ \frac{c}{R} \frac{1}{(1+ |x_o|)^{d+2s}} \int_{0}^{T} \int_{B^c_{10}(x_o)}\frac{u(t,x)}{|x-x_o|^{d+2s}} \d x \d t 
\end{split}
\end{align}
for some $c=c(d,s,\lm,\Lm, T, t_o)$.

As a result of \eqref{eq:on-diag-2}, 
 for the summation over $x_o,\,y_o \in \mathcal{S}_R$ with $|x_o - y_o| < 4$, we have
\begin{align*}
    &\sum_{\substack{x_o,\, y_o\in\mathcal{S}_R\\ |x_o-y_o|<4}} \int_{t_o}^{T}  \int_{B_{1}(x_o)} u(t,x) \int_{B_{1}(y_o)} |\Psi(t,x) - \Psi(t,y)| (1 \wedge R^{-1}|x-y|) K(t;x,y)  \d y \d x  \d t \\ \nonumber
&\qquad\le \frac{8^d}R\sum_{x_o\in\mathcal{S}_R}\int_{t_o}^{T}  \int_{B_{5}(x_o)} u(t,x) \int_{B_{5}(y_o)} |\Psi(t,x) - \Psi(t,y)| (1 \wedge R^{-1}|x-y|) K(t;x,y) \d y \d x   \d t\\
&\qquad\le \frac{c}{R}\sum_{x_o\in\mathcal{S}_R} \int_{0}^{T} \int_{B_{10}(x_o)} \frac{u(t,x)}{(1 + |x|)^{d+2s}} \d x \d t\\
&\quad\qquad + \frac{c}{R}\sum_{x_o\in\mathcal{S}_R} \frac{1}{(1+ |x_o|)^{d+2s}} \int_{0}^{T} \int_{B^c_{10}(x_o)}\frac{u(t,x)}{|x-x_o|^{d+2s}} \d x \d t  
\end{align*}
for some $c=c(d,s,\lm,\Lm, T, t_o)$.
It is straightforward to see that the summation in the first term of the right-hand side can be estimated by
\begin{align*}
\sum_{x_o\in\mathcal{S}_R} \int_{0}^{T} \int_{B_{10}(x_o)} \frac{u(t,x)}{(1 + |x|)^{d+2s}} \d x \d t \overset{\eqref{Eq:cover-balls}}{\le} 
    c\, \Vert  u \Vert_{L^1((0,T) ; L^1(\R^d;\mu))}
\end{align*}
for some $c=c(d)$.
The next goal is to show the second summation on the right-hand side can be bounded by the same quantity. In order to do that, we begin by estimating
\begin{align}\label{Eq:tail-control}
     \sum_{x_o\in\mathcal{S}_R} & \frac{1}{(1+ |x_o|)^{d+2s}} \int_{0}^{T} \int_{B_{10}^c(x_o)}\frac{u(t,x)}{|x-x_o|^{d+2s}} \d x \d t\\ \nonumber
     &= \int_{0}^{T} \int_{\R^d}\sum_{x_o\in\mathcal{S}_R}  \frac{\chi_{B_{10}^c(x_o)}(x)}{(1+ |x_o|)^{d+2s}}\frac{u(t,x)}{|x-x_o|^{d+2s}} \d x \d t\\ \nonumber
     &\le \int_{0}^{T} \int_{\R^d} u(t,x) \bigg(\sum_{x_o\in\mathcal{S}_R \setminus B_{10}(x)} \frac{1}{(1+ |x_o|)^{d+2s} |x-x_o|^{d+2s}} \bigg) \d x  \d t.
\end{align}
Moreover, for any $x\in\R^d$, the sum in the parenthesis can be estimated by 
\begin{align*}
    \sum_{x_o\in\mathcal{S} \setminus B_{10}(x)} &\frac{1}{(1+ |x_o|)^{d+2s} |x-x_o|^{d+2s}}\\ 
    &= \sum_{x_o\in\mathcal{S} \setminus B_{10}} \frac{1}{(1+ |x - x_o|)^{d+2s} |x_o|^{d+2s}} \\
    &\le c\sum_{x_o\in\mathcal{S} \setminus B_{10}} \int_{B_1(x_o)}\frac{\d y}{(1+ |x - y|)^{d+2s} |y|^{d+2s}}\\
    &\le c \int_{B^c_8}\underbrace{\bigg(\sum_{x_o\in\mathcal{S} \setminus B_{10}}\chi_{B_1(x_o)}\bigg)}_{\le C(d)\>\>\text{by}\>\eqref{Eq:cover-balls}}\frac{\d y}{(1+ |x - y|)^{d+2s} |y|^{d+2s}}\\
    &\le c \int_{B^c_8}\frac{\d y}{(1+ |x - y|)^{d+2s} |y|^{d+2s}}\\
    &\le\frac{c}{(1 + |x|)^{d+2s}}
\end{align*}
for some $c=c(d,s)$.
Here, the last estimate follows from \autoref{Lm:tail-weight-estimate}.
 Collecting these estimates in \eqref{Eq:tail-control} we obtain
     \[
     \sum_{x_o\in\mathcal{S}_R}  \frac{1}{(1+ |x_o|)^{d+2s}} \int_{0}^{T} \int_{B^c_{10}(x_o)}\frac{u(t,x)}{|x-x_o|^{d+2s}} \d x \d t\le c\,\Vert u \Vert_{L^1((0,T) ; L^1(\R^d;\mu))}
     \]
     for some $c=c(d,s)$.
     Summarizing the analysis of the \textbf{first case}, we arrive at
     \begin{align*}
         \sum_{\substack{x_o,\, y_o\in\mathcal{S}_R\\ |x_o-y_o|<4}} \int_{t_o}^{T}  \int_{B_{1}(x_o)}& u(t,x) \int_{B_{1}(y_o)} |\Psi(t,x) - \Psi(t,y)| (1\wedge R^{-1}|x-y|) K(t;x,y) \d y \d x  \d t\\
         &\le \frac{c}{R}\Vert u \Vert_{L^1((0,T) ; L^1(\R^d;\mu))}
     \end{align*}
for some $c=c(d,s, \lm,\Lm, T, t_o)$.

It remains to consider the \textbf{second case}, namely, $x_o, y_o \in \mathcal{S}_R$ with $|x_o - y_o| \ge 4$. In this case, we estimate using \eqref{eq:Psi-upper} as
\begin{align*}
    \int_{t_o}^{T} & \int_{B_1(x_o)} \int_{B_1(y_o)}  u(t,x) |\Psi(t,x) - \Psi(t,y)|  (1 \wedge R^{-1}|x-y|) K(t;x,y)  \d y \d x \d t \\
    &\le c \int_{t_o}^{T} \int_{B_1(x_o)} \frac{u(t,x)}{(1 + |x|)^{d+2s}} \int_{B_1(y_o)} \frac{1\wedge R^{-1}|x-y| }{|x-y|^{d+2s}} \d y \d x \d t \\
    &\quad + c \int_{t_o}^{T} \int_{B_1(x_o)} u(t,x) \int_{B_1(y_o)} \frac{1\wedge R^{-1}|x-y| }{(1 + |y|)^{d+2s} |x-y|^{d+2s}}  \d y \d x \d t \\
    &\le  c \frac{1\wedge R^{-1} |x_o - y_o| }{|x_o - y_o|^{d+2s}} \int_{t_o}^{T} \int_{B_1(x_o)} \frac{u(t,x)}{(1 + |x|)^{d+2s}} \d x \d t \\
    &\quad + c \bigg(\frac{1 + |x_o|}{1+|y_o|}\bigg)^{d+2s}  \frac{1\wedge R^{-1} |x_o - y_o| }{|x_o - y_o|^{d+2s}}\int_{t_o}^{T} \int_{B_1(x_o)} \frac{u(t,x)}{(1 + |x|)^{d+2s}}  \d x \d t\\
    &=: J_3(x_o,y_o) + J_4(x_o,y_o).
\end{align*}
Here, we have $c=c(d,s, \lm,\Lm, T)$.

We observe that $J_3(x_o,y_o)$ and $J_4(x_o,y_o)$ are exactly the same terms as $J_1(x_o,y_o)$ and $J_2(y_o,x_o)$ from the proof of \autoref{lemma:I1}. Hence, from here, we conclude the proof by the same arguments as in the proof of \autoref{lemma:I1}.
\end{proof}

\subsection{Proof of \autoref{Thm:0}}

First of all, we deal with the supremum estimate \eqref{eq:thm-0-1}.
Fix $\tau\in(0,T)$ and consider the new global solution
\begin{equation}\label{Eq:u-to-v}
v(t,x):=u(\tau t, \tau^{\frac{1}{2s}} x)\quad \text{in}\> \Big(0,\frac{T}{\tau}\Big)\times\R^d.
\end{equation}
Applying \autoref{Prop:bd:2} to $v$ over the cylinder $(\frac14,1]\times B_1(x_o)$, 
there exists $c>1$ depending only on $d$, $s$, $\lm$ and $\Lm$, such that
    \begin{align*}
        \sup_{(\frac14 , 1]\times B_1(x_o)} v \le  c \int_0^{1}\int_{\R^d}\frac{v(t,x)}{(1+|x-x_o|)^{d+2s}}\d x \d t.
    \end{align*}
The right-hand side integral can be estimated  by \autoref{lemma:weighted-L1-time-insensitive} in $(0,1]\times \R^d$. Consequently, we obtain
    \begin{align*}
\int_0^{1}\int_{\R^d}\frac{v(t,x)}{(1+|x-x_o|)^{d+2s}}\d x \d t & = \int_0^{1}\int_{\R^d}\frac{v(t,x+x_o)}{(1+|x|)^{d+2s}}\d x \d t\\
&\le c  \int_{\R^d}\frac{v(1,x+x_o)}{(1+|x|)^{d+2s}}\d x\\
&=c \int_{\R^d}\frac{v(1,x)}{(1+|x-x_o|)^{d+2s}}\d x,
    \end{align*}
for some $c=c(d,s,\lm,\Lm)$, since $v(t,x+x_o)$ is also a global solution in $(0,1]\times \R^d$. 
Therefore, we arrive at
    \begin{align*}
        \sup_{(\frac14,1]\times B_1(x_o)} v \le  c \int_{\R^d}\frac{v(1,x)}{(1+|x-x_o|)^{d+2s}}\d x .
    \end{align*}
Finally, scaling the above estimate back to $u$, we conclude that for any $\tau\in(0,T)$,
\[
\sup_{(\frac14\tau,\tau]\times B_{\tau^{1/2s}}(x_o)} u \le c\, \tau \int_{\R^d} \frac{u(\tau,x)}{(\tau^{1/2s}+|x-x_o|)^{d+2s} }\d x
\]
where $c=c(d,s,\lm,\Lm)$.

Next, we show the infimum estimate \eqref{eq:thm-0-2}. As before, we fix $\tau\in(0,T)$ and consider the global solution defined in \eqref{Eq:u-to-v}. 
Apply \autoref{Prop:WHI-global} to $v$ in $(0,1]\times\R^d$ and obtain that
    \begin{align*}
        c\,\inf_{(\frac34,1]\times B_1(x_o)} v  \ge   \int_{\frac14 }^{\frac12 }\int_{\R^d}\frac{v(t,x)}{(1+|x-x_o|)^{d+2s}}\d x \d t
    \end{align*}
for some $c=c(d,s,\lm,\Lm)$.
Like in the supremum estimate, the right-hand side is estimated by resorting to \autoref{lemma:weighted-L1-time-insensitive} in $(0,1]\times\R^d$. As a result, we have
    \begin{align*}
        c\,\inf_{(\frac34,1]\times B_1(x_o)}  v &\ge  \int_{\R^d}\frac{v(1,x)}{(1+|x-x_o|)^{d+2s}}\d x 
    \end{align*}
for some $c=c(d,s,\lm,\Lm)$.
Finally, we scale the above estimate back to the original solution $u$ and conclude that
     \begin{align*}
        c\,\inf_{(\frac34\tau,\tau]\times B_{\tau^{1/2s}}(x_o)} u  &\ge  \tau \int_{\R^d}\frac{u(\tau,x)}{(\tau^{1/2s}+|x-x_o|)^{d+2s}}\d x 
    \end{align*}
for some $c=c(d,s,\lm,\Lm)$. This completes the proof.

\subsection{Time-insensitive Harnack estimates fail for local solutions}
\label{subsec:counterexample}
The following result applies to local solutions in the sense of \autoref{Def:local-sol}. It shows that if an elliptic-type Harnack estimate holds, then positivity spreads in a time-insensitive fashion.

\begin{proposition}
\label{prop:slice-implies-insensitive}
Assume that the kernel $K$ satisfies \eqref{eq:Kcomp}. Let $u\ge0$ be a local, weak solution in $(-16^{2s},16^{2s})\times B_{16}$ in the sense of \autoref{Def:local-sol}. Assume that there exists $c_{\rm H}>1$, such that
\begin{equation}\label{Eq:ellip}
\sup_{B_8(y)} u(t,\cdot)\le c_{\rm H} \inf_{B_8(y)} u(t,\cdot)
\end{equation}
holds for any $y\in B_1$ and $t\in(-1,1)$.
Then, there exist some constant $c>1$ depending on  $c_{\rm H}$ and the data $\{d,s,\lm,\Lm\}$, and another constant $\gm\in(0,1)$ depending additionally on the quantity
\[
\int^{1}_{-5^{2s}}    \int_{ B^c_{2}} \frac{u(t,x)}{|x|^{d+2s}}\d x \d t,
\]
and the modulus of continuity
\[
r\mapsto \sup_{t_o\in(-1,1)} \int^{t_o}_{t_o-(3r)^{2s}}    \int_{B^c_{2}} \frac{u(t,x)}{|x|^{d+2s}}\d x \d t,
\]
such that if $u(t_o,x_o)\neq 0$ for some $(t_o,x_o)\in (-1,1)\times B_1$, then
\[
c^{-1}\sup_{B_{1}(x_o)} u(t,\cdot)  \le u(t_o,x_o)\le c \inf_{B_{1}(x_o)} u(\tau,\cdot)
\]
for any $t,\tau\in[t_o-\gm, t_o+\gm]$.
\end{proposition}

\begin{proof}
   Let us take a point $(t_o,x_o)$ from $(-1,1)\times B_1$, such that $u(t_o,x_o)\neq0$. Apart from a rescaling argument, we may assume $u(t_o,x_o)=1$. Consequently, the assumption~\eqref{Eq:ellip} indicates
\begin{equation}\label{c_H<u}
       \frac1{c_{\rm H}}\le u(t_o,x)\le c_{\rm H}, \quad\forall\, x\in B_6.
\end{equation}

The upper bound in \eqref{c_H<u} joint with  the Harnack inequality (see~\cite[Theorem 1.1]{KaWe23}) gives
\[
\sup_{[t_o-3^{2s},t_o]\times B_{3}(y)} u\le \mathfrak{B},\quad\forall\, y\in B_2,
\]
where  $\mathfrak{B}$ depends only on the data $\{d,s,\lm,\Lm\}$ and $c_{\rm H}$. This uniform bound in turn allows us to apply the continuity estimate for $u$ (see \autoref{Thm:mod-con}). We obtain that for
any $r\in(0,1]$, 
	\begin{equation*}
	\osc_{(t_o-r^{2s},t_o]\times B_{ r}(y)}u \le 2 \boldsymbol\om \Big(\frac{r}{3}\Big)^{\be}  + c  \int^{t_o}_{t_o-(3r)^{2s}}    \int_{B^c_{4}(y)} \frac{u(t,x)}{|x-y|^{d+2s}}\d x \d t,
	\end{equation*} 
where $\be\in(0,1)$ and $c>1$ are determined by the data $\{d,s,\lm,\Lm\}$, and we have taken
 $$\boldsymbol\om := 2\mathfrak{B} +\int^{t_o}_{t_o-4^{2s}}    \int_{ B^c_{4}(y)} \frac{u(x,t)}{|x-y|^{d+2s}}\d x \d t.$$
The right-hand side of this  continuity estimate can be made independent of $(t_o,y)$. Indeed, observe that since $t_o\in(-1,1)$ and $y\in B_2$ we have $(t_o-4^{2s}, t_o)\subset (-5^{2s},1)$, $B^c_{5}(y)\subset B^c_{3}$ and $|x-y|\ge\frac13 |x|$ for any $x\in B_3^c$. Hence, we estimate
\begin{align*}
    \int^{t_o}_{t_o-3^{2s}}    \int_{ B^c_{3}(y)} \frac{u(t,x)}{|x-y|^{d+2s}}\d x \d t \le  \int^{1}_{-5^{2s}}    \int_{ B^c_{2}} \frac{u(t,x)}{|x-y|^{d+2s}}\d x \d t \le  3^{d+2s} \int^{1}_{-5^{2s}}    \int_{ B^c_{2}} \frac{u(t,x)}{|x|^{d+2s}}\d x \d t.
\end{align*}
Similarly, we can estimate
\begin{align*}
\int^{t_o}_{t_o-(3r)^{2s}}    \int_{B^c_{3}(y)} \frac{u(t,x)}{|x-y|^{d+2s}}\d x \d t
\le
    3^{d+2s} \sup_{t_o\in(-1,1)}\int^{t_o}_{t_o -(3r)^{2s}}    \int_{ B^c_{2}} \frac{u(t,x)}{|x|^{d+2s}}\d x \d t.
\end{align*}
Therefore, we 
take instead a larger quantity
 $$\widetilde{\boldsymbol\om}=2\mathfrak{B} + 3^{d+2s} \int^{1}_{-5^{2s}}    \int_{ B^c_{2}} \frac{u(t,x)}{|x|^{d+2s}}\d x \d t$$
 and rewrite the continuity estimate as
\begin{equation*}
	\osc_{(t_o-r^{2s},t_o]\times B_{ r}(y)}u \le 2 \widetilde{\boldsymbol\om} \Big(\frac{r}{3}\Big)^{\be}  + c  \sup_{t_o\in(-1,1)} \int^{t_o}_{t_o-(3r)^{2s}}    \int_{B^c_{2}} \frac{u(t,x)}{|x|^{d+2s}}\d x \d t.
	\end{equation*} 
Based on the above estimate, we choose $\gm\in(0,1)$, such that
\begin{equation}\label{Eq:choose-gamma}
2 \widetilde{\boldsymbol\om} \Big(\frac{r}{3}\Big)^{\be}+ c  \sup_{t_o\in(-1,1)} \int^{t_o}_{t_o-(3r)^{2s}}    \int_{B^c_{2}} \frac{u(t,x)}{|x|^{d+2s}}\d x \d t \le \frac1{2c_{\mathrm{H}}}
\end{equation}
holds true for all $r\in(0,\gm]$. Such $\gm$ can be selected in terms of the data $\{d,s,\lm,\Lm\}$, $c_{\mathrm{H}}$, the quantity
\begin{equation}\label{Eq:choose-gamma:1}
\int^{1}_{-5^{2s}}    \int_{ B^c_{2}} \frac{u(t,x)}{|x|^{d+2s}}\d x \d t
\end{equation}
and the modulus of continuity
\begin{equation}\label{Eq:choose-gamma:2}
r\mapsto \sup_{t_o\in(-1,1)} \int^{t_o}_{t_o-(3r)^{2s}}    \int_{B^c_{2}} \frac{u(t,x)}{|x|^{d+2s}}\d x \d t.
\end{equation}
Together with the lower bound from \eqref{c_H<u}, we obtain  that
\[
\inf_{[t_o-\gm, t_o]\times B_{2}} u\ge\frac1{2 c_{\rm H}}.
\]
We use this positivity estimate together with the expansion of positivity (\autoref{lemma:expansion-of-positivity}) to conclude that
\begin{equation}\label{Eq:inf-est}
\inf_{[t_o-\frac12\gm, t_o+\frac12\gm]\times B_{2}} u\ge\frac{\eta}{2 c_{\rm H}}
\end{equation}
for some $\eta\in(0,1)$ depending only on the data $\{d,s,\lm,\Lm\}$.

Next we claim that
\begin{equation}\label{Eq:sup-est}
\sup_{[t_o-\frac14\gm, t_o+\frac14\gm]\times B_{1}} u\le \frac{4 c_{\rm H}}{\eta}.
\end{equation}
Indeed, if not, there exists some $(t_*,x_*)\in [t_o -\frac14\gm, t_o+\frac14\gm]\times B_{1}$, such that 
\[
u(t_*,x_*) = \frac{4 c_{\rm H}}{\eta}.
\]
Then, applying the previous result at $(t_*,x_*)$, modulo a proper rescaling, we obtain that
\[
\inf_{[t_*-\frac12\gm, t_* +\frac12\gm]\times B_{2}(x_*)}u \ge \frac{4 c_{\rm H}}{\eta}\cdot\frac{\eta}{2 c_{\rm H}}=2.
\]
This contradicts the fact that $(t_o,x_o)\in [t_*-\frac12\gm, t_* +\frac12\gm]\times B_{2}(x_*)$ and $u(t_o,x_o)=1$. Therefore, the desired conclusion follows from \eqref{Eq:inf-est} and \eqref{Eq:sup-est} if we choose $c$ to be $4c_{\mathrm{H}}/\eta$ and redefine $\gm/4$ as $\gm$.
\end{proof}

\begin{remark}
    If we know apriori that $u\le M$ for some $M>0$, then \eqref{Eq:choose-gamma:1} and \eqref{Eq:choose-gamma:2} can be estimated more explicitly, and hence the quantity $\gm$ of \autoref{prop:slice-implies-insensitive} can be chosen in terms of $M$, $c_\mathrm{H}$ and the data. In this setting, the key qualitative information is that if an elliptic-type estimate \eqref{Eq:ellip} holds for any $y\in B_1$ and $t\in(-1,1)$, then, by a simple chain-argument, $u$ must be either positive or vanishing in $B_1\times(-1,1)$.
\end{remark}

The following example shows that the waiting-time is necessary in the parabolic Harnack inequality for solutions in bounded domains.

\begin{example}
Let $u$ be the weak solution to
\begin{align*}
\begin{cases}
\partial_t u + (-\Delta)^s u &= 0 ~~ \text{ in } (-16^{2s} , 16^{2s}) \times B_{16},\\
u &= 0 ~~ \text{ in } (-16^{2s},0] \times (\R^d \setminus B_{16}),\\
u &= 1 ~~ \text{ in } (0,16^{2s}) \times (\R^d \setminus B_{16}),\\
u(-16^{2s}) &= 0 ~~ \text{ in } \R^d.
\end{cases}
\end{align*}
Clearly, by the maximum principle (see \cite[Lemma 5.2]{KaWe23}) it must be $u \equiv 0$ in $(-16^{2s},0] \times B_{16}$ and $0<u \le 1$ in $(0,16^{2s}) \times B_{16}$. \autoref{prop:slice-implies-insensitive} indicates that \eqref{Eq:ellip} cannot hold for $u$.
\end{example}

\section{Improved weak Harnack inequality for local solutions}
\label{sec:improved-weak-Harnack}

\subsection{Expansion of positivity}
A central ingredient of our proof is the following expansion of positivity result from \cite[Proposition 4.1]{Lia24b}. 

Note that the proof only requires the kernel $K$ to satisfy the upper bound in \eqref{eq:Kcomp} and \eqref{eq:coercive}.

\begin{proposition}
\label{lemma:expansion-of-positivity}
Assume that the kernel $K$ satisfies the upper bound in \eqref{eq:Kcomp} and \eqref{eq:coercive}.
Let $u \ge 0$ be a weak supersolution in $(t_o , t_o + 4R^{2s}) \times B_{4R}(x_o)$ in the sense of \autoref{Def:local-sol}. Assume that for some $\alpha \in (0,1]$ and $k > 0$ it holds
\begin{align*}
| \{ u(t_o,\cdot) \ge k \} \cap B_R(x_o)| \ge \alpha |B_R(x_o)|.
\end{align*}
Then, there exist constants $\eta \in (0,1)$ and $p>1$ depending only on the data $\{d,s,\lm,\Lm\}$, such that
\begin{align*}
u \ge \eta \alpha^p k ~~ \text{ in } (t_o + R^{2s}, t_o + 4R^{2s}) \times B_{2R}(x_o).
\end{align*}
\end{proposition}

\begin{proof}
    The proof presented in \cite{Lia24b} employed the lower bound in \eqref{eq:Kcomp} because of \cite[Lemma 3.4]{Lia24b}. However, \cite[Lemma 3.3]{Lia24b} continues to hold only assuming the upper bound in \eqref{eq:Kcomp} and \eqref{eq:coercive}. Therefore, starting from the given initial measure theoretical information, one can propagate it, i.e. for some $\varep,\,\dl\in(0,1)$ depending only on the data and $\al$, we have
    \[
    | \{ u(t,\cdot) \ge \varep k \} \cap B_{4R}(x_o)| \ge\tfrac12 4^{-d} \alpha |B_{4R}(x_o)|
    \]
    for all
    \[
    t\in(t_o,t_o+\dl(4R)^{2s}).
    \]
    Based on this measure theoretical estimate, we can apply the weak Harnack inequality which is established under the assumption \eqref{eq:coercive} and the upper bound of \eqref{eq:Kcomp} in~\cite{FeKa13}. Once a pointwise estimate is obtained, we can repeat this argument and further propagate it and reach the desired pointwise estimate, see \cite[\S~4.1]{Lia24b}.
\end{proof}

\subsection{Propagation of positivity in $L^1_{\loc}$}

With the help of the energy decay estimate in \autoref{lemma:energy-decay}, we can establish the following lemma.

\begin{lemma}\label{Lm:inf-L1}
Assume that the kernel $K$ satisfies the upper bound in \eqref{eq:Kcomp} and \eqref{eq:coercive}.
Let $u\ge0$ be a weak supersolution in $(0,4 )\times B_4$ satisfying
\begin{align}
\label{eq:inf-L1-ass}
\int_0^{2} \int_{B_2} u^{q} \d x \d t \le \mathbf{C}
\end{align}
for some $q \in (1,1 + \frac{2s}{d})$ and $\mathbf{C}>1$. Moreover, assume that for some $c_o > 0$ we have
\begin{align}\label{eq:L1-c0}
\int_{B_{\frac12}} u(0,x) \d x \ge c_o.
\end{align}
Then, there exists $\tau \in (0,\frac14)$, depending only on the data $\{d,s,\lm,\Lm\}$, $q$, $c_o$ and $\mathbf{C}$, such that 
\begin{align*}
\inf_{t\in(0, \tau] } \int_{B_1} u(t,x) \d x \ge \tfrac12 c_o.
\end{align*}
\end{lemma}

\begin{proof}
Take a cutoff function $\phi$ supported in $B_{\frac23}$ with $\phi \equiv 1$ in $B_{\frac12}$. Testing the weak formulation for $u$ with $\phi$, we obtain for any $\tau \in (0,1)$ that
\begin{align*}
\int_{B_1} u(\tau,x) \phi(x) \d x &\ge \int_{B_1} u(0,x) \phi(x) \d x\\ 
&\qquad - \int_0^\tau \int_{\R^d} \int_{\R^d} (u(t,x) - u(t,y))(\phi(x) - \phi(y)) K(t;x,y) \d y \d x \d t \\
&\ge  c_o - I.
\end{align*} 
Here, we used the assumption \eqref{eq:L1-c0} to insert $c_o$ and denoted by $I$ the second term on the right-hand side. Let us estimate $I$ from above. First, relying \eqref{eq:inf-L1-ass}, we resort to \autoref{lemma:energy-decay} and obtain that
\begin{align*}
\int_0^{\tau} & \int_{B_{1}} \int_{B_{1}} (u(t,x) - u(t,y))(\phi(x) - \phi(y)) K(t;x,y) \d y \d x \d t \\
& \le c \int_0^{\tau} \int_{B_{1}} \int_{B_{1}}  |u(t,x) - u(t,y)| K(t;x,y) |x-y| \d y \d x \d t \le c\, \tau^{1-\frac1q}
\end{align*}
for some $c=c(d,s,\lm,\Lm,q,\mathbf{C})$.
Moreover, since $u \ge 0$ we have
\begin{align*}
\int_0^{\tau} & \int_{B_{\frac23}} \int_{  B^c_{1}} (u(t,x) - u(t,y))(\phi(x) - \phi(y)) K(t;x,y) \d y \d x \d t \\
&= \int_0^{\tau} \int_{B_{\frac23}} \int_{ B^c_{1}} (u(t,x) - u(t,y)) \phi(x) K(t;x,y) \d y \d x \d t \\
&\le \int_0^{\tau} \int_{B_{\frac23}} u(t,x) \phi(x) \left(  \int_{ B^c_{1}} K(t;x,y) \d y \right) \d x \d t \\
&\le c \int_0^{\tau} \int_{B_{1}} u(t,x) \d x \d t \\
&\le c\, \tau^{1-\frac1q} \left(\int_0^1 \int_{B_1} u^{q}(t,x) \d x \d t \right)^{\frac{1}{q}} \le c \, \tau^{1-\frac{1}{q}}
\end{align*}
for some $c=c(d,s,\lm,\Lm,q,\mathbf{C})$, where we used the upper bound in \eqref{eq:Kcomp}. Moreover, we used H\"older's inequality in the last step, and applied \eqref{eq:inf-L1-ass}. Altogether, we deduce that
\begin{align*}
\int_{B_1} u(\tau,x) \phi(x) \d x 
&\ge  c_o - c \, \tau^{1-\frac{1}{q}}\ge\tfrac12 c_o,
\end{align*} 
provided we choose $\tau\le(\frac{c_o}{2c})^{\frac{q}{q-1}}$.
\end{proof}

\subsection{A pointwise positivity estimate}

In this section, we combine what has been obtained previously and obtain a weak Harnack type estimate assuming an $L^\sig$--{\it a priori} bound.
\begin{lemma}\label{Lm:2nd-alt}
Assume that the kernel $K$ satisfies the upper bound in \eqref{eq:Kcomp} and \eqref{eq:coercive}.
Let $u\ge0$ be a weak supersolution in $(0,8) \times B_8$. Assume that there exist constants $\sig\in(0,1)$ and $c_o,\,c_1>0$, such that
\begin{equation}\label{L-sig-condition}
    \int_0^{2}\int_{B_2} u^{\sigma}(t,x)\,\d x \d t\le c_1
\end{equation}
and
\begin{equation}\label{L-1-condition}
    \int_{B_1}u(0,x)\d x\ge c_o.
\end{equation}
Then, there exists $\eta\in(0,1)$ depending only on the data $\{d,s,\lm,\Lm\}$, $\sig$, $q$, $c_o$ and $c_1$, such that
\begin{equation*}
    u\ge \eta\quad\text{a.e. in}\quad [2,8]\times B_1.
\end{equation*}
\end{lemma}

\begin{proof}
From \autoref{Lm:Lq-Lsig} and condition~\eqref{L-sig-condition} we obtain that for any $q\in(1,1+\frac{2s}d)$,
\begin{align*}
    \int_{0}^{1} \int_{B_{1}} u^q(t,x) \d x \d t\le \mathbf{C},
\end{align*}
where $\mathbf{C}$ depends only on the data $\{d,s,\lm,\Lm\}$, $\sig$, $q$ and $c_1$. On the other hand, according to \autoref{Lm:inf-L1} and \eqref{L-1-condition} there exists some $\tau\in(0,\frac14)$, depending on $\{d,s,\lm,\Lm\}$, $q$, $c_o$ and $\mathbf{C}$, such that
\begin{align*}
    \inf_{t\in[0,\tau]}\int_{B_1}u(t,x)\d x\ge\tfrac{1}{2} c_o.
\end{align*}
As a consequence of the last two estimates, there must be some $t_*\in[0,\tau]$, such that
\[
\int_{B_{1}}u^q(t_*,x)\d x\d t\le \mathbf{C}\quad\text{and}\quad \int_{B_1} u(t_*,x)\d x \ge\tfrac12 c_o.
\]
Next, based on the above display, for some $\xi>0$ to be fixed, we can compute by H\"older's inequality
    \begin{align*}
       \tfrac12 c_o&\le \int_{B_1}u(t_*,x)\d x=\int_{B_1}u(t_*,x)\chi_{\{u\ge \xi\}}\d x+\int_{B_1}u(t_*,x)\chi_{\{u< \xi\}}\d x\\
        &\le \bigg(\int_{B_1}u^q(t_*,x)\d x\bigg)^{\frac1q}|\{u(t_*,\cdot)\ge\xi\}\cap B_1|^{1-\frac1q}+\xi |B_1|\\
        &\le \mathbf{C} |\{u(t_*,\cdot)\ge\xi\}\cap B_1|^{1-\frac1q}+\xi |B_1|.
    \end{align*}
Let us choose $\xi=\frac{c_o}{4|B_1|}$ in the previous estimate. Then, we arrive at
\[
|\{u(t_*,\cdot)\ge\xi\}\cap B_1|\ge \frac{c_o}{4\mathbf{C}}.
\]
Now, we use this measure estimate and apply the expansion of positivity in \autoref{lemma:expansion-of-positivity} to conclude that
\[
u\ge\eta\quad\text{a.e. in}\quad [t_*+1, t_*+2]\times B_2
\]
for some $\eta$ depending only on the data $\{d,s,\lm,\Lm\}$,  $c_o$ and $\mathbf{C}$.
Note that $t_*\in[0,\frac14]$, and hence the above estimate actually yields that
\[
u\ge\eta \quad\text{a.e. in}\quad [\tfrac54, 2]\times B_2.
\]
Once pointwise estimate is obtained, we can further apply the expansion of positivity to propagate it up to the top of the domain.
This finishes the proof.
\end{proof}

\subsection{Proof of the improved weak Harnack inequality}
We are now in a position to conclude the proof of \autoref{thm:improved-weak-Harnack}.
Assume without loss of generality that the supremum on the left-hand side is taken at $t_o$. Moreover, let us stipulate $(t_o,x_o) = (0,0)$, $R = 1$, and $\|u(t_o,\cdot)\|_{L^1(B_R(x_o))}=1$.
Otherwise, we consider the rescaled function
\[
v(t,x)=\frac{u(R^{2s} (t-t_o), R (x-x_o))}{\|u(t_o, \cdot)\|_{L^1(B_R(x_o))}},
\]
which is a weak supersolution to \eqref{Eq:1:1} in $(0,8)\times B_8$.

Let $p>1$ be the number determined in \autoref{lemma:expansion-of-positivity}. Now, we consider two alternatives: either there exist $t_1\in[0,1]$ and $k>1$, such that
\[
|\{u(t_1,\cdot)> k\}\cap B_2|\ge k^{-\frac1{p+1}}|B_2|,
\]
or for any $t\in[0,1]$ and any $k>1$, we have
\[
|\{u(t,\cdot)> k \}\cap B_2|< k^{-\frac1{p+1}}|B_2|.
\]

If the first alternative occurs, we apply the expansion of positivity in \autoref{lemma:expansion-of-positivity} and obtain
\begin{equation*}
    u\ge \eta k^{\frac1{p+1}}>\eta\quad\text{a.e. in}\quad (t_1 + 1, t_1 + 4)\times B_2
\end{equation*}
for some $\eta$ depending on the data $\{d,s,\lm,\Lm\}$.
For any $t_1\in[0,1]$, we can always conclude from the last line that 
\begin{equation*}
    u\ge \eta \quad\text{a.e. in}\quad (2,  4)\times B_2.
\end{equation*}
Once this pointwise estimate is obtained, we can keep using the expansion of positivity and further propagate it up to the top of the domain, which allows us to conclude the proof.

It remains to deal with the second alternative. To begin with, let $\sig:=\frac1{2p}$ and observe that for any $t\in[0,1]$,
\begin{align*}
    \int_{B_1} u^\sig(t,x)\d x
    &=\int_{0}^{\infty} |\{u(t,\cdot)>k\}\cap B_1| \d k^\sig\\
    &= \int_{1}^{\infty} |\{u(t, \cdot)>k\}\cap B_1| \d k^\sig+\int_{0}^{1} |\{u(t, \cdot)>k\}\cap B_1| \d k^\sig\\
    &\le |B_2| \int_1^{\infty}k^{-\frac{1}{p+1}}\d k^\sig+|B_1|\\
    &\le c+|B_1|
\end{align*}
for some $c = c(d,p) > 0$, where we used that $\frac{1}{\sigma(p+1)} = \frac{2p}{p+1} > 1$ since $p > 1$.
This means \eqref{L-sig-condition} of \autoref{Lm:2nd-alt} is fulfilled with $c_1=c+|B_1|$. At the same time, \eqref{L-1-condition} is granted with $c_o=1$. Therefore, \autoref{Lm:2nd-alt} yields some $\eta\in(0,1)$ depending only on the data $\{d,s,\lm,\Lm\}$, such that
\begin{equation*}
    u\ge\eta\quad\text{a.e. in}\quad (2,8)\times B_1.
\end{equation*}
Therefore, no matter which alternative occurs, we arrive at a similar pointwise estimate.
The proof is complete after reverting to the original solution.

\appendix
\section{Representation of solution}
\label{sec:appendix}

This appendix deals with solutions to the Cauchy problem
\begin{align}
\label{eq:inhom-Cauchy}
    \begin{cases}
        \partial_t u - \mathcal{L}_t u &= 0 ~~ \text{ in } (\eta,T) \times \R^d,\\
        u(\eta) &= f ~~ \text{ in } \R^d,
    \end{cases}
\end{align}
where $\mathcal{L}_t$ is a time-dependent operator \eqref{eq:op}, $f \in L^2(\R^d)$ and $\eta \in [0,T)$. Analogously, for $\xi\in(0,T]$ we also consider the dual problem
\begin{align}
\label{eq:dual-inhom-Cauchy}
    \begin{cases}
        \partial_t u + \mathcal{L}_t u &= 0 ~~ \text{ in } (0,\xi) \times \R^d,\\
        u(\xi) &= f ~~ \text{ in } \R^d.
    \end{cases}
\end{align}
In particular, we establish the existence of the fundamental solutions together with some of their key properties and representations of solution to both of these Cauchy problems. Previously, this was only known for time-independent operators via a semigroup or a probability approach. Here, we provide an entirely analytic argument.

\begin{proposition}
    \label{prop:representation}
    Assume that the kernel $K$ satisfies the upper bound in \eqref{eq:Kcomp} and \eqref{eq:coercive}. Let $\eta \in [0,T)$. Then, there exists a function $(t,x,y) \mapsto p_{\eta,t}(x,y)$, such that for any $f \in L^2(\R^d)$ the function
    \begin{align}
    \label{eq:representation}
        (t,x) \mapsto \int_{\R^d} p_{\eta,t}(x,y) f(y) \d y
    \end{align}
    is the unique solution to \eqref{eq:inhom-Cauchy} in the sense of \autoref{def:Cauchy} and attaining its initial datum in the sense of $L^2_{\loc}(\R^d)$. 
    Moreover, it holds for any $x,y \in \R^d$ and $0 \le \eta < \tau < t < T$ that
    \begin{align}
    \label{eq:p-properties}
        p_{\eta,t}(x,y) \ge 0,  ~~ \int_{\R^d} p_{\eta,t}(x,y) \d y \le 1, ~~ p_{\eta,t}(x,y) = \int_{\R^d} p_{\tau,t}(x,z) p_{\eta,\tau}(z,y) \d z.
    \end{align}
   
   Likewise, for $0<t<\xi\le T$ there exists a function $(t,x,y) \mapsto \hat{p}_{\xi,t}(x,y)$, such that for any $f \in L^2(\R^d)$ the function
        \begin{align}
    \label{eq:dual-representation}
        (t,x) \mapsto \int_{\R^d} \hat{p}_{\xi,t}(x,y) f(y) \d y
    \end{align}
    solves the dual problem \eqref{eq:dual-inhom-Cauchy}. Moreover,  the following symmetry holds
    \begin{align}
    \label{eq:dual-relation}
        p_{\eta,t}(x,y) = \hat{p}_{t,\eta}(y,x) \qquad \forall\, x,y \in \R^d, ~~ \forall\, 0 \le \eta < t \le T.
    \end{align}
\end{proposition}

Note that \eqref{eq:p-properties} and \eqref{eq:dual-relation} are required in order to apply the results in \cite{KaWe23}.

\subsection{Existence and uniqueness for Cauchy problems}
The proof of \autoref{prop:representation} goes in several steps. First, we establish the existence and uniqueness of a weak solution to \eqref{eq:inhom-Cauchy}.

\begin{lemma}
\label{lemma:ex-sg}
    Assume that the hypothesis of \autoref{prop:representation} holds. For any $f \in L^2(\R^d)$, there exists a unique weak solution to the Cauchy problem \eqref{eq:inhom-Cauchy} in the sense of \autoref{def:Cauchy}. 
    The solution is denoted by $(t,x) \mapsto P_{\eta,t}f(x)$.
\end{lemma}

\begin{proof}
Let $B_n$ be the ball of radius $n\in\N$ centered at the origin. Consider the boundary value problem
\begin{equation}\label{Eq:problem-n}
\begin{cases}
        \partial_t u - \mathcal{L}_t u &= 0 ~~ \text{ in } (\eta,T) \times B_n,\\
        u(\eta) &= f\chi_{B_n} ~~ \text{ in } \R^d,\\
        u &= 0 ~~ \text{ in } (\eta, T) \times (\R^d \setminus B_n).
    \end{cases}
\end{equation}
By \cite[Theorem 5.3]{FKV15}, this problem admits a unique solution $u_n$ in the function space $C((\eta,T); L^2(B_n))\cap L^2(\eta,T; H_{0}^s(B_n))$ in the sense that
\begin{align}\label{weak-form-n}
        -\int_\eta^T \int_{\R^d} u_n(t,x) \partial_t \phi(t,x) \d x \d t + \int_\eta^T \cE^{(t)}(u_n(t),\phi(t)) \d t = 0
    \end{align}
   for any $\phi \in H^{1}((\eta,T);L^2(B_n))\cap  L^2(\eta,T;H^s(\R^d))$ with $\supp(\phi) \Subset (\eta,T) \times B_n$.
In addition, $u$ attains its initial datum in the $L^2$ sense, i.e. $u(t) \to f \chi_{B_n}$ in $L^2(\R^d)$. Moreover, 
\begin{align*}
    \sup_{\eta<t<T}\int_{\R^d} |u_n(t)|^2\d x +2\int_\eta^T\mathcal{E}^{(t)}(u_n(t),u_n(t))\d t\le \|f\|^2_{L^2(\R^d)}.
\end{align*}
In particular, this means that the family $\{u_n:\,n\in\N\}$ is uniformly bounded in $L^2((\eta,T);H^s(\R^d))$ and in $L^{\infty}((\eta,T);L^2(\R^d))$. Moreover, by the comparison principle (see \cite[Lemma 5.2]{KaWe23}), $\{u_n:\,n\in\N\}$ forms an increasing sequence. Therefore, we can extract a subsequence $u_{n^{\prime}}$ and identify some $u\in L^2((\eta,T);H^s(\R^d))$, such that $u_{n^{\prime}}\nearrow u$ a.e. in $(\eta,T)\times\R^d$, in $L^2((\eta,T);L^2(\R^d))$, and weakly in $L^2((\eta,T);H^s(\R^d))$. Consequently, we can pass to the limit in \eqref{weak-form-n}. Moreover, it holds $u \in L^{\infty}((\eta,T);L^2(\R^d))$ by the Banach-Saks theorem.

Next, we show that $u$ attains the initial datum $f$ in the $L^2_{\loc}(\R^d)$ sense. To see this, let $f_{k} \in C^{\infty}_c(\R^d)$ be a mollification of   $f \chi_{B_k}$ such that $f_{k} \to f$ in $L^2(\R^d)$, as $k \to \infty$. We let $k \in \N$ be fixed and test the weak formulation of $u_n$ with $u_n - f_{k}$, which is possible if $n \in \N$ is large enough. Then, we obtain for any $t \in (\eta,T)$
\begin{align*}
\int_{\R^d} |u_n(t) - f_{k}|^2 \d x \le \int_{\R^d}  |f - f_{k}|^2 \d x - 2\int_{\eta}^t \cE^{(t)}(u_n(t),u_n(t) - f_k) \d t.
\end{align*}
Here, we also used that $(u_n - f_k) (\partial_t u_n) = \frac{1}{2}  \partial_t (u_n - f_k)^2$ and a proper time-mollification. By the Cauchy-Schwartz inequality,
\begin{align*}
-\cE^{(t)}(u_n(t),u_n(t) - f_k) &= -\cE^{(t)}(u_n(t) - f_k,u_n(t) - f_k)   -\cE^{(t)}(f_k,u_n(t) - f_k) \\
&\le -\frac{1}{2}\cE^{(t)}(u_n(t) - f_k,u_n(t) - f_k)  + 4\cE^{(t)}(f_k,f_k) \\
&\le 4\cE^{(t)}(f_k,f_k),
\end{align*}
and therefore we have upon taking the limit $n \to \infty$ and by the triangle inequality
\begin{align*}
\Vert u(t) - f \Vert_{L^2(\mathcal{K})}^2 \le 2 \Vert f - f_{k} \Vert_{L^2(\R^d)}^2 + C (t - \eta) [f_k]_{H^s(\R^d)}^2.
\end{align*}
Altogether, taking the limit $t \searrow \eta$, we obtain for any $k \in \N$ and any compact set $\mathcal{K} \subset \R^n$ that
\begin{align*}
\lim_{t \searrow \eta} \Vert u(t) - f \Vert_{L^2(\mathcal{K})}^2 \le 2 \Vert f - f_{k} \Vert_{L^2(\R^d)}^2,
\end{align*}
which implies that $u(t) \to f$ in $L^2(\mathcal{K})$ as $t \searrow \eta$ due to the construction of $f_k$.

To prove that solutions to \eqref{eq:inhom-Cauchy} are unique, we let $u$ be an arbitrary solution attaining its initial data $f$ in the $L^2_{\loc}(\R^d)$ sense, and take $u \z^2$ as a testing function, where $\z \in C^{\infty}_c(B_{\frac32R})$ with $0 \le \z \le 1$, $\z \equiv 1$ in $B_R$, and $|\nabla \z| \le 4/ R$. It holds that for any $\tau \in (\eta,T)$,
\begin{align*}
    \int_{B_{2R}}\z^2 u^2(\tau,\cdot) \d x +2\int_\eta^{\tau}\int_{\R^d}\int_{\R^d}(\z(x)u(t,x)-\z(y)u(t,y))^2 K(t;x,y)\d x\d y \d t\\    
    =\int_{B_{2R}}\z^2 f^2 \d x + 2\int_\eta^{\tau} \int_{\R^d}\int_{\R^d}u(t,x)u(t,y)(\z(x)-\z(y))^2 K(t;x,y)\d x\d y \d t.
\end{align*}
Here, we also used that $u(\partial_t u) = \frac{1}{2} \partial_t (u^2)$  modulo a proper time mollification, and integrated by parts. Hence, we deduce
\begin{align*}
    \int_{B_{2R}}\z^2 u^2(\tau,\cdot) \d x &\le \int_{B_{2R}}\z^2 f^2 \d x + 2\int_\eta^{\tau} \int_{\R^d}\int_{\R^d}u(t,x)u(t,y)(\z(x)-\z(y))^2 K(t;x,y)\d x\d y \d t \\
    &\le \int_{B_{2R}}\z^2 f^2 \d x + \frac{C}{R^{2s}} \int_{\eta}^{\tau} \int_{\R^d} u^2 \d x \d t,
\end{align*}
for some $C=C(d,s,\Lm)$. Here, we used the upper bound of $K$ in \eqref{eq:Kcomp}, Young's inequality and the following well-known fact
\begin{align*}
    \int_{\R^d} \frac{(\z(x) - \z(y))^2}{|x-y|^{d+2s}} \d y \le c  \int_{B_{2R}(x)} \frac{R^{-2}}{|x-y|^{d-(2-2s)}} \d y + c \int_{B_{2R}^c(x)} \frac{1}{|x-y|^{d+2s}} \d y \le \frac{C}{R^{2s}}.
\end{align*}
Thus, using that $u \in L^2((\eta,T);L^2(\R^d))$ and taking the limit $R \nearrow \infty$, we deduce that any weak solution to \eqref{eq:inhom-Cauchy} satisfies
\begin{align}
\label{eq:L2-est}
    \sup_{\tau \in (\eta,T)}\int_{\R^d} u^2(\tau,\cdot) \d x \le \int_{\R^d} f^2 \d x.
\end{align}
In particular, if $u_1,u_2 \in L^{\infty}((\eta,T);L^2(\R^d))\cap L^2((\eta, T); H^s(\R^d))$ are two solutions to \eqref{eq:inhom-Cauchy} with initial datum $f \in L^2(\R^d)$, then $u_1-u_2$ is a solution to \eqref{eq:inhom-Cauchy} with zero initial datum. Hence, \eqref{eq:L2-est} implies that $u_1 \equiv u_2$, so solutions to \eqref{eq:inhom-Cauchy} are unique. 
\end{proof}

\begin{remark}\label{Rmk:max-pri}\upshape
    If $f\in L^{2}(\R^d)\cap L^{\infty}(\R^d)$, then the solution $u_n$ to the boundary value problem \eqref{Eq:problem-n} satisfies $\|u_n\|_{L^{\infty}(\R^d)}\le \|f\|_{L^{\infty}(\R^d)}$ due to the maximum principle. Consequently, by construction, we have $\|P_{\eta,t}f\|_{L^{\infty}(\R^d)}\le \|f\|_{L^{\infty}(\R^d)}$.
\end{remark}
\subsection{Existence of fundamental solution}
Next, we establish the existence of the fundamental solution and prove the representation formula in \eqref{eq:representation}.

\begin{lemma}\label{Lm:fund-sol}
    Assume that we are in the setup of \autoref{prop:representation}. Then, there exists a function $(t,x,y) \mapsto p_{\eta,t}(x,y)$ such that 
    \begin{align}
    \label{eq:representation-2}
        P_{\eta,t}f(x) = \int_{\R^d} p_{\eta,t}(x,y) f(y) \d y.
    \end{align}
\end{lemma}

\begin{proof}
    Note that for any $\eta,t,x$, the map $A \mapsto P_{\eta,t} \chi_{A}(x)$, where $A \subset \R^d$ is a Borel set, is a Borel measure. We claim that for any set $A$ with $\chi_A \in L^1(\R^d)$ it holds
    \begin{align}
    \label{eq:abs-cont}
        \Vert P_{\eta,t} \chi_A \Vert_{L^{\infty}(\R^d)} \le  c \Vert \chi_A \Vert_{L^2(\R^d)}
    \end{align}
    for some $c = c(\eta,t) > 0$. This estimate implies that $A \mapsto P_{\eta,t} \chi_A(x)$ is absolutely continuous with respect to the Lebesgue measure for any $\eta,t,x$. Hence, by the Radon-Nikodym theorem, there exists a density, which we denote by $y \mapsto p_{\eta,t}(x,y)$ such that
    \begin{align*}
        P_{\eta,t} \chi_A(x) = \int_{A} p_{\eta,t}(x,y) \d y = \int_{\R^d} p_{\eta,t}(x,y) \chi_A(y) \d y.
    \end{align*}
    Note that by the parabolic weak maximum principle (see \cite[Lemma 5.2]{KaWe23}), we have $P_{\eta,t}\chi_A(x) \ge 0$ and also $p_{\eta,t}(x,y) \ge 0$.
    Finally, note that any nonnegative function $f \in L^2(\R^d)$ can be approximated by a monotone sequence of simple functions, and hence, we deduce \eqref{eq:representation-2}  for such functions $f$ by monotone convergence and the stability of weak solutions (see again \cite[Lemma 5.3]{KaWe23}). For general $f \in L^2(\R^d)$, we decompose $f = f_+ - f_-$ and apply the previous argument two $f_{\pm}$ separately, hence proving \eqref{eq:representation-2}.

    It remains to prove \eqref{eq:abs-cont}. 
   We apply the local boundedness estimate (see \autoref{Prop:bd}) and H\"older's inequality to deduce that for any $x_o \in \R^d$
   \begin{align*}
       \Vert P_{\eta,t}f \Vert_{L^{\infty}(B_{(t-\eta)^{1/2s}}(x_o))} &\le c \int_{\eta}^t \int_{\R^d} \frac{P_{\eta,\tau} f(x)}{ (1 + |x-x_o|)^{d+2s}} \d x \d \tau\\
       &\le c \sup_{\tau \in (\eta,T)}\Vert P_{\eta,\tau} f \Vert_{L^2(\R^d)},
   \end{align*}
   where $c = c(\eta,t)$. Using \eqref{eq:L2-est} and the arbitrariness of $x_o$, we deduce
   \begin{align*}
       \Vert P_{\eta,t} f \Vert_{L^{\infty}(\R^d)} \le  c \Vert f \Vert_{L^2(\R^d)},
   \end{align*}
   which in particular implies \eqref{eq:abs-cont} and concludes the proof.
\end{proof}

\subsection{Properties of the fundamental solution}
    The purpose of this subsection is to show that the fundamental solution $p_{\eta,t}(x,y)$ identified in \autoref{Lm:fund-sol} satisfies the three properties listed in \eqref{eq:p-properties}.  
    
    The first and second property in \eqref{eq:p-properties}, namely nonnegativity and integrability of $p_{\eta,t}$, follow from the maximum principle (see \cite[Lemma 5.2]{KaWe23}). This is immediately clear for the nonnegativity. To see the integrability, note that according to \autoref{Rmk:max-pri}, we have
\[
\int_{\R^d} \chi_\mathcal{K}(y)p_{\eta,t}(x,y) \d y =P_{\eta, t}\chi_{\mathcal{K}}(x)\le 1
\]
for any compact set $\mathcal{K}$ and any $x\in\R^d$. Then, this implies that 
$$
\sup_{x\in\R^d}\int_{\R^d} p_{\eta,t}(x,y) \d y\le1.
$$
    
    The third property in \eqref{eq:p-properties} (semigroup property) follows from the observation that by the uniqueness of solutions for any $f \in L^2(\R^d)$ and $\eta < \tau < t$
    \begin{align*}
       P_{\eta,t}f = P_{\tau,t} (P_{\eta,\tau} f),
    \end{align*}
    which implies using \eqref{eq:representation} and Fubini's theorem that
    \begin{align*}
        \int_{\R^d} p_{\eta,t}(x,y) f(y) \d y &= \int_{\R^d} p_{\tau,t}(x,y) P_{\eta,\tau} f(y) \d y \\
        & = \int_{\R^d} \left( \int_{\R^d} p_{\tau,t}(x,y) p_{\eta,\tau}(y,z) f(z) \d z \right) \d y \\
        &= \int_{\R^d} \left( \int_{\R^d} p_{\tau,t}(x,z) p_{\eta,\tau}(z,y) \d z \right) f(y) \d y.
    \end{align*}
    Hence, we deduce the second property by a density argument.

\subsection{Symmetry of the fundamental solution}
The goal of this subsection is to prove \eqref{eq:dual-relation}. The proof uses the same idea as the proof of \autoref{lemma:weighted-L1-time-insensitive}, but this time, instead of the pointwise upper bound \eqref{eq:Psi-upper}, we can use that $P_{\eta,t} f \in L^{\infty}((\eta,T);L^2(\R^d))$.

Denote by $\hat{P}_{T,t}f$ the solution to the dual Cauchy problem, i.e. $\partial_t u + \mathcal{L}_t u = 0$ with $u(T) = f$. Its existence and uniqueness is ensured by \autoref{lemma:ex-sg}. Moreover, the existence of this dual problem's fundamental solution $\hat{p}_{T,t}(x,y)$ is ensured by \autoref{Lm:fund-sol}.  Let us consider nonnegative $f, \, g \in L^2(\R^d)$. Let $\xi_R \in C_c^{1}(B_{2R})$ be such that $0 \le \xi_R \le 1$ and $\xi_R \equiv 1$ in $B_R$, and $|\nabla \xi_R| \le 4 /R$. Moreover, we fix $t_o\in(\eta,T)$ and consider arbitrary $\tau_o,\,\tau_1$ satisfying $t_o \le \tau_o < \tau_1 < T$. Introduce a sequence of   functions $\zeta_k\in C_c^{1}(0,T)$, $k\in\N$ satisfying
\begin{equation*}
\left\{
    \begin{array}{cc}
        \zeta_k \to \chi_{[\tau_o, \tau_1]},  \quad
         \partial_t \zeta_k \to \delta_{\tau_o} - \delta_{\tau_1}\quad \text{as}\>\> k\to\infty , \\ [5pt]
         \displaystyle\supp(\zeta_k) \subset \Big[\frac{t_o+\tau_o}{2}, \frac{T +\tau_1}{2}\Big].
    \end{array}\right.
\end{equation*}
Proceeding as in the proof of \autoref{lemma:weighted-L1-time-insensitive} and replacing the roles of $(u,\Psi)$ by $( \hat P_{T,\tau} g, P_{\eta,\tau} f)$, we  obtain 
\begin{align}\label{Eq:P-P-hat}
    \bigg| \int_{\eta}^{T} \partial_t \zeta_k(t) \int_{\R^d} \xi_R(x) P_{\eta,t} f (x) \hat{P}_{T,t}g(x) \d x \d t \bigg| \le c\, \textbf{I}_R + c\, \textbf{II}_R,
\end{align}
where
\begin{align*}
    \textbf{I}_R &:= \int_{t_o}^{T} \iint_{(B_R^c \times B_R^c)^c} P_{\eta,t} f (x) \big| \hat{P}_{T,t}g(x) - \hat{P}_{T,t}g(y) \big| (1 \wedge R^{-1}|x-y|) K(t;x,y) \d y \d x \d t, \\
    \textbf{II}_R &:= \int_{t_o}^{T} \iint_{(B_R^c \times B_R^c)^c} \hat{P}_{T,t} g (x) \big| P_{\eta,t}f(x) - P_{\eta,t}f(y) \big| (1 \wedge R^{-1}|x-y|) K(t;x,y) \d y \d x \d t.
\end{align*}
We claim that $\textbf{I}_R , \, \textbf{II}_R \to 0$ as $R \to \infty$. This follows in a similar way as shown in \autoref{lemma:I1} and  \autoref{lemma:I2}, and relies on the fact that $P_{\eta,t} f ,\, \hat{P}_{T,t} g \in L^{\infty}((\eta,T);L^2(\R^d))$. Once this claim is established, we let $k\to\infty$ and then $R\to\infty$ in \eqref{Eq:P-P-hat} and obtain that
\begin{align*}
    t\mapsto\int_{\R^d} P_{\eta,t} f (x) \hat{P}_{T,t}g(x) \d x
\end{align*}
is a constant function on $(\eta, T)$. Choose in the above identity $f = \chi_A$ and $g = \chi_B$ for some compact sets $A,\, B \subset \R^d$. Recall from \autoref{lemma:ex-sg} that $P_{\eta,t} f\to f$ as $t\searrow\eta$ and $\hat{P}_{T,t} g \to g$ as $t\nearrow T$ in the sense of $L^2_{\loc}(\R^d)$. Thus, taking limit as $t\nearrow T$ and as $t\searrow\eta$ in the above constant function, we find that
\begin{align*}
    \int_{\R^d} P_{\eta,T} f(x) g(x) \d x = \int_{\R^d} \hat{P}_{T,\eta} g(x) f(x) \d x.
\end{align*}
Meanwhile, evoking the representation formulas for $P_{\eta,T} f$ and $\hat{P}_{T,t} g$ from \autoref{Lm:fund-sol}, we deduce
\begin{align*}
    \int_{A} \int_B p_{\eta,T}(x,y) \d y \d x = \int_{A} \int_B \hat{p}_{T,\eta}(y,x) \d y \d x,
\end{align*}
which implies $p_{\eta,T}(x,y) = \hat{p}_{T,\eta}(y,x)$. This proves \eqref{eq:dual-relation}, as desired

Hence, it remains to prove that $\textbf{I}_R , \, \textbf{II}_R \to 0$ as $R \to \infty$. We only show  $\textbf{I}_R \to 0$ as $R \to \infty$, as the argument for $\textbf{II}_R$ is analogous. In order for that, we reproduce some arguments from \autoref{lemma:I2}. In fact, we split $\textbf{I}_{R/2}$ into five parts:
\begin{align*}
    \textbf{I}_1 &= \int_{\eta}^T\int_{B_{R}}P_{\eta,t} f (x)\int_{B_{R}}\big| \hat{P}_{T,t}g(x) - \hat{P}_{T,t}g(y) \big| (1 \wedge R^{-1}|x-y|) K(t;x,y) \d y \d x \d t,\\
    \textbf{I}_2 &= \int_{\eta}^T\int_{B_{R/2}} P_{\eta, t}f(x)\hat{P}_{T,t} g(x)\int_{B_R^c}  \frac{1 \wedge R^{-1}|x-y|}{|x-y|^{d+2s}}  \d y \d x \d t,\\
    \textbf{I}_3 &= \int_{\eta}^T\int_{B_{R/2}} P_{\eta, t}f(x) \int_{B_R^c}  \hat{P}_{T,t} g(y)\frac{1 \wedge R^{-1}|x-y|}{|x-y|^{d+2s}}  \d y \d x \d t,\\
    \textbf{I}_4 &= \int_{\eta}^T\int_{B_R^c} P_{\eta, t}f(x)\hat{P}_{T,t} g(x)\int_{B_{R/2}}  \frac{1 \wedge R^{-1}|x-y|}{|x-y|^{d+2s}}  \d y \d x \d t,\\
    \textbf{I}_5 &= \int_{\eta}^T\int_{B^c_{R}} P_{\eta, t}g(x) \int_{B_{R/2}}  \hat{P}_{T,t} f(y)\frac{1 \wedge R^{-1}|x-y|}{|x-y|^{d+2s}}  \d y \d x \d t.
\end{align*}
Repeating what has been done in \autoref{lemma:I2} and applying H\"older's inequality, we obtain
\begin{align*}
    \textbf{I}_2 + \textbf{I}_4 &\le \frac{c}{R^{2s}}\Vert P_{\eta,t} f \hat{P}_{T,t} g \Vert_{L^1((\eta,T) \times \R^d)} \le \frac{c}{R^{2s}}\Vert P_{\eta,t} f  \Vert_{L^2((\eta,T) \times \R^d)} \Vert \hat{P}_{T,t} g \Vert_{L^2((\eta,T) \times \R^d)},
\end{align*}
for some $c=c(d,s)$.
Likewise,
\begin{align*}
    \textbf{I}_3 + \textbf{I}_5 &\le c \Vert P_{\eta,t} f\Vert_{L^{2}((\eta,T) ; L^1(B_R))} \bigg( \int_{\eta}^T \bigg(\int_{B_R^c} \frac{\hat{P}_{T,t} g(x)}{|x|^{d+2s}} \d x \bigg)^2 \d t \bigg)^{1/2} \\
    &\quad + c \Vert \hat{P}_{T,t} g \Vert_{L^{2}((\eta,T) ; L^1(B_R))}  \bigg( \int_{\eta}^T \bigg(\int_{B_R^c} \frac{P_{\eta,t} f(x)}{|x|^{d+2s}} \d x \bigg)^2 \d t \bigg)^{1/2} \\
    &\le \frac{c}{R^{2s}} \Vert P_{\eta,t} f\Vert_{L^{2}((\eta,T) \times \R^d)} \Vert \hat{P}_{T,t} g \Vert_{L^{2}((\eta,T)  \times \R^d)}
\end{align*}
for some $c=c(d,s)$.

It remains to treat $\textbf{I}_1$. To this end, let us recall the set $\mathcal{S}_R$ satisfying \eqref{Eq:cover-balls} and use a covering argument to obtain
\begin{align*}
    \textbf{I}_1 &\le\sum_{x_o,\, y_o\in\mathcal{S}_R } \int_{t_o}^{T}   \int_{B_{1}(x_o)} P_{\eta,t} f (x)  \int_{B_{1}(y_o)} \big| \hat{P}_{T,t}g(x) - \hat{P}_{T,t}g(y) \big| \\
    &\qquad\qquad\qquad\qquad\qquad\qquad\qquad\qquad\qquad\cdot (1 \wedge R^{-1}|x-y|) K(t;x,y) \d y \d x \d t.
\end{align*}
According to the distance between $x_o$ and $y_o$, we need to consider \textbf{two cases}.
In the \textbf{first case}, i.e. when $|x_o - y_o| < 4$, we use \autoref{Lm:Lq-Lsig} and \autoref{lemma:energy-decay} with $q \in (1,1+\frac{2s}{d})$ and $\sigma = 1$, to obtain that
\begin{align*}
     \frac1R \int_{t_o}^{T}   \int_{B_{5}(x_o)} \int_{B_{5}(y_o)} &\big| \hat{P}_{T,t}g(x) - \hat{P}_{T,t}g(y) \big|  K(t;x,y) |x-y| \d y \d x \d t \\
     &\le \frac{c}{R} \int_{\eta}^T \int_{B_{10}(x_o)} \hat{P}_{T,t}g(x) \d x \d t,
\end{align*}
for some $c=c(d,s,\lm,\Lm, \eta, t_o)$. Moreover, applying H\"older's inequality, the elementary inequality 
$$
\sum_{i=1}^{n}a_i^{\frac12}\le n^{\frac12}\Big(\sum_{i=1}^{n}a_i\Big)^{\frac12}
$$

and then \eqref{Eq:cover-balls}, we have
\begin{align*}
\sum_{x_o\in\mathcal{S}_R} \Vert \hat{P}_{T,t}g \Vert_{L^1((\eta,T) \times B_{10}(x_o))} &\le (T-\eta)^{\frac12}|B_{10}|^{\frac12}\sum_{x_o\in\mathcal{S}_R}  \bigg[ \int_{\eta}^T \int_{B_{10}(x_o)} |\hat{P}_{T,t}g(x)|^2 \d x \d t\bigg]^{\frac12}\\
&\le c\, \#(\mathcal{S}_R)^{\frac12}\bigg[ \sum_{x_o\in\mathcal{S}_R} \int_{\eta}^T \int_{B_{10}(x_o)} |\hat{P}_{T,t}g(x)|^2 \d x \d t\bigg]^{\frac12}\\
&\le c\, R^{\frac{d}{2}} \Vert \hat{P}_{T,t}g \Vert_{L^2((\eta,T) \times \R^d)}
\end{align*}
for some $c=c(d,T)$.
In addition, by the local boundedness estimate of \autoref{Prop:bd} in $B_{2R}$ and H\"older's inequality
\begin{align*}
    \sup_{(t_o,T) \times B_1(x_o)} P_{\eta,t} f &\le \sup_{(t_o,T) \times B_{2R}} P_{\eta,t} f\\ 
    &\le 
    \frac{c}{R^{d}} \Vert P_{\eta,t} f \Vert_{L^1((\eta,T) \times B_{4R})} + c R^{2s} \int_{\eta}^T \int_{B_{4R}^c}  \frac{P_{\eta,t} f(x)}{ |x|^{d+2s}} \d x \d t \\
    &\le  \frac{c}{R^{\frac{d}{2}}} \Vert P_{\eta,t} f \Vert_{L^2((\eta,T) \times \R^d)}
\end{align*}
for some $c=c(d,s,\lambda,\Lambda)$.
Consequently, combining the previous estimates,  we obtain
\begin{align*}
&\sum_{\substack{x_o,\, y_o\in\mathcal{S}_R\\ |x_o-y_o|<4}} \int_{t_o}^{T}\int_{B_{1}(x_o)} P_{\eta,t} f (x) \int_{B_{1}(y_o)}\big| \hat{P}_{T,t}g(x) - \hat{P}_{T,t}g(y) \big| (1 \wedge R^{-1}|x-y|) K(t;x,y) \d y \d x \d t \\ 
&\quad\le \frac{c}{R^{1+\frac{d}{2}}} \Vert P_{\eta,t} f \Vert_{L^{2}(\eta,T) \times \R^d)}  \sum_{x_o\in\mathcal{S}_R} \Vert \hat{P}_{T,t}g \Vert_{L^1((\eta,T) \times B_{10}(x_o))}\\ 
&\quad \le \frac{c}{R} \Vert P_{\eta,t} f \Vert_{L^2(\eta,T) \times \R^d)} \Vert \hat{P}_{T,t}g \Vert_{L^2((\eta,T) \times \R^d)}
\end{align*}
for some $c=c(d,s,\lambda,\Lambda,\eta,t_o)$.
In the \textbf{second case}, i.e. when $|x_o - y_o| \ge 4$, by a similar argument as in \eqref{J1-J2}, we obtain
\begin{align*}
     &\int_{t_o}^{T}\int_{B_{1}(x_o)} P_{\eta,t} f (x) \int_{B_{1}(y_o)}\big| \hat{P}_{T,t}g(x) - \hat{P}_{T,t}g(y) \big| (1 \wedge R^{-1}|x-y|) K(t;x,y) \d y \d x \d t \\
     &\le c\frac{1 \wedge R^{-1} |x_o-y_o|}{|x_o - y_o|^{d+2s}}   \int_{\eta}^T \int_{B_1(x_o)}  P_{\eta,t} f (x)  \hat{P}_{T,t}g(x) \d x \d t \\
     &\quad + c\frac{1 \wedge R^{-1} |x_o-y_o|}{|x_o - y_o|^{d+2s}}  \int_{\eta}^T \bigg( \int_{B_1(x_o)}  P_{\eta,t} f (x) \d x \bigg) \bigg( \int_{B_1(y_o)}  \hat{P}_{T,t}g(y) \d y \bigg)  \d t \\
     &=: c\, J_1(x_o,y_o) +c\, J_2(x_o,y_o).
\end{align*}
For $J_1$, we can proceed exactly as for $J_1$ in the proof of \autoref{lemma:I1}. Namely, we use \eqref{Eq:aug-est-1}  and H\"older's inequality to estimate
\begin{align*}
    \sum_{\substack{x_o,\, y_o\in\mathcal{S}_R\\ |x_o-y_o|\ge4}}  J_1(x_o,y_o) &= \sum_{x_o\in\mathcal{S}_R} \int_{\eta}^T \int_{B_1(x_o)}  P_{\eta,t} f (x)  \hat{P}_{T,t}g(x) \d x \d t \bigg(\sum_{\substack{y_o\in\mathcal{S}_R\\ |x_o-y_o|\ge4}}\frac{1 \wedge R^{-1} |x_o-y_o|}{|x_o - y_o|^{d+2s}} \bigg)\\
    &\le \Big(c\frac{\mathfrak{C}(R,s)}R + \frac{c}{R^{2s}}\Big) \|P_{\eta,t} f \cdot \hat{P}_{T,t}g\|_{L^1((\eta,T)\times\R^d)}\\
    &\le c\Big(\frac{\mathfrak{C}(R,s)}R + \frac{1}{R^{2s}}\Big) \|P_{\eta,t} f \|_{L^2((\eta,T)\times\R^d)}\|\hat P_{\eta,t} g \|_{L^2((\eta,T)\times\R^d)}.
\end{align*}
For $J_2$, we first apply Cauchy-Schwartz's and Young's inequality, and then \eqref{Eq:aug-est-1} and \eqref{Eq:cover-balls} to get
\begin{align*}
    \sum_{\substack{x_o,\, y_o\in\mathcal{S}_R\\ |x_o-y_o|\ge4}}  J_2(x_o,y_o) &\le c \sum_{\substack{x_o,\, y_o\in\mathcal{S}_R\\ |x_o-y_o|\ge4}} \frac{1 \wedge R^{-1} |x_o-y_o|}{|x_o - y_o|^{d+2s}}  \Vert P_{\eta,t} f \Vert_{L^2((\eta,T) \times B_1(x_o))}^2 \\
    &\quad + c \sum_{\substack{x_o,\, y_o\in\mathcal{S}_R\\ |x_o-y_o|\ge4}}  \frac{1 \wedge R^{-1} |x_o-y_o|}{|x_o - y_o|^{d+2s}}  \Vert \hat{P}_{T,t} g \Vert_{L^2((\eta,T) \times B_1(y_o))}^2 \\
    &\le c \Big(\frac{\mathfrak{C}(R,s)}R + \frac{1}{R^{2s}} \Big)\\
    &\quad\cdot\bigg( \sum_{x_o\in\mathcal{S}_R}  \Vert P_{\eta,t} f \Vert_{L^2((\eta,T) \times B_1(x_o))}^2 + \sum_{y_o\in\mathcal{S}_R}  \Vert \hat{P}_{T,t} g \Vert_{L^2((\eta,T) \times B_1(y_o))}^2  \bigg) \\
    &\le c \Big(\frac{\mathfrak{C}(R,s)}R + \frac{1}{R^{2s}} \Big) \bigg(  \Vert P_{\eta,t} f \Vert_{L^2((\eta,T) \times \R^d)}^2 + \Vert \hat{P}_{T,t} g \Vert_{L^2((\eta,T) \times \R^d)}^2 \bigg).
\end{align*}
The proof is complete after combining these two cases and letting $R\to\infty$.


\end{document}